\documentclass[10pt,reqno]{amsart}
\usepackage{latexsym,amsmath}
\usepackage{enumerate}
\usepackage{amsfonts}
\usepackage{amssymb}
\usepackage{latexsym}

\usepackage[T1]{fontenc}
\usepackage{fourier}
\usepackage{bbm}

\usepackage{amsfonts,amsmath}
\usepackage{color} 
\usepackage{fullpage}
\usepackage{latexsym,amsmath}

\usepackage{bbm}

\DeclareMathOperator{\tr}{tr}

\numberwithin{equation}{section}

\renewcommand{\vec}[1]{\boldsymbol{#1}}

\newcommand{\dk}{d_{k-\mathrm{SAT}}}

\newcommand\KL[2]{D_{\mathrm{KL}}\bc{{{#1}\|{#2}}}}

\newcommand\SIGMA{\vec\sigma}

\newtheorem{definition}{Definition}[section]

\newtheorem{claim}[definition]{Claim}

\newtheorem{theorem}[definition]{Theorem}
\newtheorem{lemma}[definition]{Lemma}
\newtheorem{proposition}[definition]{Proposition}
\newtheorem{corollary}[definition]{Corollary}

\newtheorem{fact}[definition]{Fact}

\newcommand\sign{\mathrm{sign}}

\newcommand\id{\mathrm{id}}

\newcommand\PHI{\vec\Phi}

\newcommand\cA{\mathcal{A}} 
\newcommand\cB{\mathcal{B}} 
\newcommand\cC{\mathcal{C}} 
\newcommand\cD{\mathcal{D}} 
\newcommand\cF{\mathcal{F}} 
 
\newcommand\cE{\mathcal{E}}

\newcommand\cQ{\mathcal{Q}} 
 
\newcommand\cS{\mathcal{S}} 
 
\newcommand\cI{\mathcal{I}}

\newcommand\cM{\mathcal{M}}

\newcommand\cX{\mathcal{X}}

\def\cC{{\mathcal C}}
\def\cE{{\mathcal E}}

\newcommand\eps{\varepsilon}

\newcommand\Var{\mathrm{Var}} 

\DeclareMathOperator{\Erw}{\mathrm E}
\DeclareMathOperator{\pr}{\mathrm P}

\newcommand{\vecone}{\vec{1}}

\newcommand{\Po}{{\rm Po}} 
\newcommand{\Bin}{{\rm Bin}}

\newcommand{\bink}[2] {{{#1}\choose {#2}}}

\newcommand\bc[1]{\left({#1}\right)} 
\newcommand\cbc[1]{\left\{{#1}\right\}} 
\newcommand\bcfr[2]{\bc{\frac{#1}{#2}}} 
\newcommand{\bck}[1]{\left\langle{#1}\right\rangle} 
\newcommand\brk[1]{\left\lbrack{#1}\right\rbrack} 
\newcommand\scal[2]{\bck{{#1},{#2}}} 
 
\newcommand\abs[1]{\left|{#1}\right|}

\newcommand{\whp}{w.h.p.} 
 
\newcommand{\stacksign}[2]{{\stackrel{\mbox{\scriptsize #1}}{#2}}} 
\newcommand{\tensor}{\otimes}

\newcommand{\Bollobas}{Bollob\'as}

\newcommand\Lem{Lemma}
\newcommand\Prop{Proposition}
\newcommand\Thm{Theorem}

\newcommand\Cor{Corollary}
\newcommand\Sec{Section}

\begin{document}

\title{The number of satisfying assignments of random regular $k$-SAT formulas}

\author[Coja-Oghlan and Wormald]{Amin Coja-Oghlan \and Nick Wormald}
\thanks{The research leading to these results has received funding from the European Research Council under the European Union's Seventh Framework
			Programme (FP/2007-2013) / ERC Grant Agreement n.\ 278857--PTCC}
\date{\today} 
 
\address{Amin Coja-Oghlan, {\tt acoghlan@math.uni-frankfurt.de}, Goethe University, Mathematics Institute, 10 Robert Mayer St, Frankfurt 60325, Germany.}

\address{Nick Wormald,  {\tt nick.wormald@monash.edu}, Monash University, Department of Mathematical Sciences, VIC 3800, Australia.}

\maketitle

\begin{abstract}
\noindent
Let $\PHI$ be a random $k$-SAT formula in which every variable occurs precisely $d$ times positively and $d$ times negatively.
Assuming that $k$ is sufficiently large and that $d$ is slightly below the critical degree where the formula becomes unsatisfiable with high probability,
we determine the limiting distribution of the logarithm of the number of satisfying assignments.

\hfill MSC: 60C05, 05C80.
\end{abstract}

\section{Introduction}

\noindent
In order to study random instances of constraint satisfaction problems it is key to get a handle on the number of solutions.
In fact, in many examples such as $k$-colorability in random graphs  the best current estimates of the threshold
for the existence of solutions derive from calculating the second moment of the number of solutions~\cite{nature,Danny}.
Furthermore, if the number of solutions is sufficiently concentrated, then typical properties of random solutions
as well as the geometry of the set of solutions can be studied by way of the `planted model', an easily accessible distribution~\cite{Barriers}.
However, {prior to this work} the limiting distribution of the number of solutions has not been determined  precisely in any of the standard examples
of random constraint satisfaction problems.

In this paper we show how the limiting distribution of the number of solutions can be obtained by combining the second moment method
with a subtle application of the ``small subgraph conditioning'' technique. 
The concrete problem that we deal with is the {\em random regular $k$-SAT problem}.
In this model there are $n$ Boolean variables $x_1,\ldots,x_n$ and $m=2dn/k$ Boolean clauses of length $k$.
We always assume that $k$ divides $2dn$.
The random formula $\PHI_n(d,k)$ is obtained by choosing without replacement for each variable $x_i$ precisely $d$ out of the $km$ available literal slots
where $x_i$ appears positively and another $d$ slots where $x_i$ appears negatively.
Let $Z$ be the number of satisfying assignments of $\PHI=\PHI_n(d,k)$.

For $k$ exceeding a certain constant $k_0$ an explicit literal degree $\dk$ is known such that~\cite{kSAT}
	\begin{align}\label{eqSAT}
	\liminf_{n\to\infty}\pr\brk{\PHI\mbox{ is satisfiable}}&>0&\mbox{ if }d<\dk,\\
	\lim_{n\to\infty}\pr\brk{\PHI\mbox{ is satisfiable}}&=0&\mbox{ if }d>\dk.
		\nonumber
	\end{align}
While the precise formula  is cumbersome, in the limit of large $k$ we have
	\begin{align}\label{eqThreshold}
	2\dk/k=2^k\ln2-k\ln2/2-(1+\ln2)/2+\eps_k,\quad\mbox{where }\lim_{k\to\infty}\eps_k=0.
	\end{align}
Our main result determines the limiting distribution of $ Z$ for degrees $d$ almost (but not quite) matching $\dk$.

\begin{theorem}\label{Thm_SAT}
There exists a constant $k_0$ such that for all $k\geq k_0$ and $d>0$ such that
	\begin{align}\label{eqAssumptionOnd}
	2d/k\leq2^k\ln 2-k\ln 2/2-4
	\end{align}
the following is true.
Let $q=q(k)\in(0,1)$ be the unique solution to the equation
	\begin{equation}\label{eqq}
	2q=1-(1-q)^k
	\end{equation}
{Moreover, for $l\geq1$ and $0\leq t\leq l$ let
	\begin{align}\label{eqLambdaDeltaThm}
	\lambda_{l,t}&=\frac1{2l}\bink lt
		\bcfr{(k-1)(d-1)}{2}^l\bcfr d{d-1}^{t}
		&
		\delta_{l,t}&=(-1)^t(2q-1)^l
	\end{align}}
and let $(\Lambda_{l,t})_{l,t}$ be a family of independent Poisson variables with $\Erw[\Lambda_{l,t}]=\lambda_{l,t}$.
Then the random variable
	\begin{align}\label{eqW}
	W=\prod_{l\geq1}\prod_{t\leq l}(1+\delta_{l,t})^{\Lambda_{l,t}}\exp(-\lambda_{l,t}\delta_{l,t})
	\end{align}
satisfies $\Erw[W^2]<\infty$ and
	\begin{align}\label{eqZtoW}
	Z\cdot\frac{(4q(1-q))^{dn+\frac12}}{2^n(1-(1-q)^k)^{m+\frac12}}\ \stacksign{$n\to\infty$}\longrightarrow\  W\qquad\mbox{in distribution}.
			\end{align}
\end{theorem}

It is not difficult to verify that
	\begin{equation}\label{eqFirstMomentBig}
	n\ln2+\bc{m+\frac12}\ln(1-(1-q)^k)-\bc{dn+\frac12}\ln(4q(1-q))=\Omega(n)
	\end{equation}
	 for $d$ satisfying (\ref{eqAssumptionOnd}), {and additionally that $\Erw|\ln W|<\infty$}.
{Hence, (\ref{eqZtoW}) and (\ref{eqFirstMomentBig}) imply that $\ln Z=\Omega(n)$ \whp\ for such $d$.}
In particular, we obtain

\begin{corollary}
For $k\geq k_0$ and $d$ satisfying (\ref{eqAssumptionOnd}) we have
	$\lim_{n\to\infty}\pr\brk{\PHI\mbox{ is satisfiable}}=1$.
\end{corollary}

\noindent
The constant $4$ in (\ref{eqAssumptionOnd}) is not optimal.
In fact, a truncated second moment argument as in~\cite{Lenka} in combination with an argument similar to~\cite{victor}
might extend the above results up to the exact condensation threshold of the regular $k$-SAT problem,
although both steps would require substantial technical work.

\subsection*{Related work}
Random regular $k$-SAT instances were first studied by Rathi, Aurell, Rasmussen and Skoglund~\cite{Rathi} via the second moment method.
They proved that $\liminf_{n\to\infty}\pr\brk{\PHI\mbox{ is satisfiable}}>0$ for degrees $d$ close to $\dk$.
The latter was determined by Coja-Oghlan and Panagiotou~\cite{kSAT} by a second moment argument
that incorporates ``Survey Propagation'', a technique from statistical physics~\cite{MPZ}.
A closely related paper by Ding, Sly and Sun~\cite{DSS1} studies the regular $k$-NAESAT problem,
	which asks for a satisfying assignment whose inverse is satisfying as well.
In fact, Ding, Sly and Sun have an argument based on Fourier analysis that shows that the NAE-satisfiability probability is
not just bounded away from $0$ but actually converges to $1$ (in contrast to~(\ref{eqSAT})).
{Recently Sly, Sun and Zhang~\cite{SSZ} extended this argument to calculate the expectation of the $n$th root of the number of NAE-solutions.
This is quite a difficult problem due to a phenomenon known as ``replica symmetry breaking'' in physics~\cite{MPZ}.
However, \cite{SSZ} does not determine the limiting distribution.}

Conceptually the regular $k$-SAT model is simpler than the better known
uniform model where a specific number of clauses are drawn uniformly and independently.
This is because the local structure of regular formulas fluctuates less as each variable has precisely $d$ positive and $d$
negative occurrences and the total number of cycles of a fixed length is bounded \whp\
In the case of uniformly random $k$-SAT formulas Frieze and Wormald~\cite{FriezeWormald}
used the second moment method to determine the $k$-SAT threshold in the case that $k=k(n)\to\infty$ as $n\to\infty$.
Moreover, for clause lengths $k$ that remain fixed as $n\to\infty$ Achlioptas and Moore~\cite{nae} significantly improved the previous
lower bound on the satisfiability threshold by applying the second moment method to the number of	NAE-solutions.
Working with ``balanced'' assignments instead,  Achlioptas and Peres~\cite{yuval} improved the NAE-lower bound by a factor of two.
This left an additive gap of $\Theta(k)$ between the lower bound and an upper bound of  Kirousis, Kranakis, Krizanc and Stamatiou~\cite{KKKS}.
Coja-Oghlan and Panagiotou~\cite{kSAT,SAT} narrowed the gap to a function that tends to $0$ in the limit of large $k$ by a second moment
argument inspired by Survey Propagation.
Finally, Ding, Sly and Sun~\cite{DSS3} determined the precise satisfiability threshold in uniformly random formulas for large enough $k$
via a second moment argument that fully rigorises the Survey Propagation calculations.

We prove \Thm~\ref{Thm_SAT} by combining the second moment argument from~\cite{SAT} with small subgraph conditioning.
This method was originally developed to prove that random regular graphs of degree at least three are Hamiltonian \whp~\cite{RobinsonWormald}.
Using Skorokhod's representation theorem,
Janson~\cite{Janson} showed that small subgraph conditioning can be used to obtain limiting distributions.
However, Janson's result does not seem to apply directly in our context.
Instead, we perform a variance analysis along the lines of \cite{RobinsonWormald} for a family of random variables
that count satisfying assignments with certain peculiar properties.

Based on an early version of the present paper, the technique explained in Section~\ref{Sec_overview} was used by Rassmann~\cite{R} to analyse the number of 2-colourings of random $k$-uniform hypergraphs.

\subsection*{Notation and preliminaries}
Throughout the paper we tacitly assume that $n$ is sufficiently large, that $k$ exceeds a sufficiently large constant $k_0$
and that $d$ satisfies (\ref{eqAssumptionOnd}).
We encode the Boolean values `true' and `false' by $1$ and $-1$, respectively.
Moreover, we extend truth assignments $\sigma:\{x_1,\ldots,x_n\}\to\{\pm1\}$ to the set of literals by letting $\sigma(\neg x_i)=-\sigma(x_i)$.
We use $O$-notation with respect to both $n$ and $k$, 
with the convention that $O(1)$, $o(1)$ etc.\ always refer to the limit {as} $n\to\infty$.
For a number $l$ and an integer $h\geq0$ we write
	$$l^{\underline h}=\prod_{0\leq i<h}(l-i)$$
for the falling factorial; in particular, $l^{\underline 0}=1$.
Further, with the convention {$\ln0=-\infty$, $0\ln0=0\ln\frac00=0$}, we recall that the Kullback-Leibler divergence of two probability distributions
$(p_x)_{x\in\cX},(q_x)_{x\in\cX}$ on a finite set $\cX$ is
	\begin{align}\label{eqKL}
	\KL qp&=\sum_{x\in\cX}q_x\ln\frac{q_x}{p_x}\in[0,\infty].
	\end{align}
Finally, we denote the scalar product of vectors $\xi,\eta$ by $\scal\xi\eta$ and we write $\vecone$ for the vector with all entries equal to one (in any dimension).

\section{Overview}\label{Sec_overview}

\noindent
The basic insight behind small subgraph conditioning is that the fluctuations of $\ln Z$ can be attributed to
the number of certain small sub-structures of the random formula $\PHI$.
To elaborate, we rephrase the definition of $\PHI$ by modifying {what is essentially} a bijection model due to B{\'e}k{\'e}ssy,  B{\'e}k{\'e}ssy and  Koml\'os~\cite{BBK} in the context of matrices with given line sums. With the incorporation of signs, it becomes the following:
we view $\PHI$ as a uniformly random bijection 
	\begin{equation}\label{eqConfModel}
	[m]\times[k]\to\{x_1,\ldots,x_n\}\times[d]\times\{\pm1\},\qquad(i,j)\mapsto\PHI[i,j].
	\end{equation}
Thus, (\ref{eqConfModel}) maps each clause index $i\in[m]$ and each position $j\in[k]$ in that clause
to a Boolean variable $x\in\{x_1,\ldots,x_n\}$, an index $h\in[d]$ {(denoting which of the $d$ copies of the literal is used),} and a sign $s\in\{\pm1\}$ {} indicating whether the variable
appears as a {positive} or as a negative literal.
In terms of propositional formulas, the triple $\PHI[i,j]$ corresponds to the literal $x$ if $s=1$ and $\neg x$ if $s=-1$.
Let us write $\partial(i,j)=\partial_{\PHI}(i,j)$ for the first and $\sign(i,j)=\sign_{\PHI}(i,j)$ for the last component of $\PHI[i,j]$.
Then an assignment $\sigma:\{x_1,\ldots,x_n\}\to\{\pm1\}$ satisfies $\PHI$ iff
	$\min_{i\in[m]}\max_{j\in[k]}\sign(i,j)\sigma(\partial(i,j))=1$.
Thus, we can write
	$$Z=\sum_{\sigma:\{x_1,\ldots,x_n\}\to\{\pm1\}}\prod_{i=1}^m\brk{1-\prod_{j=1}^k\frac{1-\sign(i,j)\sigma(\partial(i,j))}{2}}.$$
Because (\ref{eqConfModel}) is a bijection each variable appears  precisely $2d$ times in total
in the corresponding propositional formula, namely $d$ times positively and $d$ times negatively.
Further, for a literal $l$ and an index $h\in[d]$ we let $\partial(l,h)=\partial_{\PHI}(l,h)$ denote the pair $(i,j)\in[m]\times[k]$ such that
$\PHI[i,j]=(x,h,1)$ if $l=x$ and $\PHI[i,j]=(x,h,-1)$ if $l=\neg x$.

It is natural to represent $\PHI$ by a bipartite {multigraph}, the {\em factor graph} $G(\PHI)$.
It has vertices $[m]$ corresponding to the clauses and vertices $\{x_1,\ldots,x_n\}$ representing the Boolean variables.
For each pair $(i,j)\in[m]\times[k]$ we insert an edge between $i$ and the variable $x$ such that $\PHI[i,j]\in\{x\}\times[d]\times\{\pm1\}$.
Additionally, we annotate the edge by $\sign(i,j)\in\{\pm1\}$.
Of course, $G(\PHI)$ may well have multiple edges.

Because the factor graph is sparse and random, standard arguments show that it is unlikely to contain many short cycles.
Hence, if we explore the factor graph from a randomly chosen root clause for some bounded number $2l$ of steps, then 
we will typically see a ``deterministic'' tree in which each clause has degree $k$ and every variable has $d$ positive and $d$ negative occurrences.
However, a bounded number of clauses will {take part in any cycles} of length at most $2l$.
As it will be important to keep track of the literal signs traversed along the cycle,
for a given $s=(s_2,\ldots,s_{2l+1})\in\{\pm1\}^{2l}$ we let $C_s=C_s(\PHI)$ be the number of cycles of length $2l$ in which the initial literal has sign $s_2$,%
the second one has sign $s_3$, etc. 
{(The starting index is chosen for convenient index arithmetic.)} 
{We call $s$ the {\em sign pattern} of the cycle.}
Moreover, to avoid {overcounting} we always deem the clause with the smallest index the starting point of the cycle,
	and the cycle is oriented towards the slot of that clause with the smaller index.
Formally, given $l\geq1$ and a {sign pattern} $s=(s_{2},\ldots,s_{2l+1})\in\{\pm1\}^{2l}$, let $C_s$ 
be the number of sequences $(i_2,j_2),\ldots,(i_{2l+1},j_{2l+1})\in[m]\times[k]$ such that
	\begin{description}
	\item[CY1] $i_2=i_{2l+1}=\min\{i_2,\ldots,i_{2l}\}$ and $i_2,\ldots,i_{2l}$ are pairwise distinct
	\item[CY2] $i_{t+1}=i_t$ if $t\in[2l+1]$ is odd,
	\item[CY3] $\partial(i_t,j_t)=\partial(i_{t+1},j_{t+1})$ if $t\in[2l+1]$ is even but $\partial(i_{2},j_{2}),\ldots,\partial(i_{2l},j_{2l})$ are pairwise distinct,
	\item[CY4] we have $j_2<j_{2l+1}$,
	\item[CY5] $\sign(i_t,j_t)=s_{i_tj_t}$ for all $t\in[2l+1]$.
	\end{description}
Moreover, for $\ell\geq1$ let $\cF_\ell=\cF_{\ell,n}(d,k)$ be the $\sigma$-algebra generated by the random variables
$C_s$ with $s\in\bigcup_{l\leq\ell}\{\pm1\}^{2l}$.

By the standard decomposition of the variance, we can write for any $\ell\geq1$
	\begin{align}\label{eqVarDecomp}
	\Erw[Z^2]-\Erw[Z]^2&=
		\Erw[\Erw[Z|\cF_\ell]^2-\Erw[Z]^2]+\Erw[\Erw[Z^2|\cF_\ell]-\Erw[Z|\cF_\ell]^2].
	\end{align}
The term $\Erw[\Erw[Z|\cF_\ell]^2-\Erw[Z]^2]$ accounts for the amount of variance induced by the fluctuations
of the number of cycles of length at most $2\ell$.
Given the number of cycles of length at most $2\ell$, the conditional variance $\Var[Z|\cF_\ell]=\Erw[\Erw[Z^2|\cF_\ell]-\Erw[Z|\cF_\ell]^2]$ remains. 
Generally, small subgraph conditioning is based on showing that
	\begin{align}\label{eqVarDecomp2}
	\lim_{\ell\to\infty}\limsup_{n\to\infty}
		\Erw\brk{\frac{\Erw[Z^2|\cF_\ell]-\Erw[Z^2]}{\Erw[Z]^2}}=0.
	\end{align}
In other words, in the limit of large $\ell$ and $n$, {with $n$ growing much faster than $\ell$,} the second summand in (\ref{eqVarDecomp}) is negligible.
Thus, once we condition on the number of short cycles the variance is tiny.
If so, then the limiting distribution of $\ln Z$ is just the limit of
	$\ln\Erw[Z|\cF_\ell]$ as $n,\ell\to\infty$, which is determined by the joint distribution of the number of short cycles.

Due to the combinatorial nature of the regular $k$-SAT problem a direct attempt at proving (\ref{eqVarDecomp2}) leads to fairly unpleasant calculations.
Indeed, {the inherent asymmetry of the Boolean values `true' and `false' causes the formula for the second moment of $Z$ to involve} implicit parameters
that we find tedious to track directly (although it might be possible).
Similar issues arise in other random constraint satisfaction problems as well.
Further, they also appear in the formula for the $k$-SAT threshold in the regular $k$-SAT problem~\cite{kSAT}.

{In this case, we are able to develop a version of the} small subgraph conditioning argument that does not require such extensive calculations. 
To this end, we decompose $Z$ into a sum of contributions that are tractable by fairly simple combinatorial considerations.
Specifically, let $\Sigma=\{\pm1\}^k\setminus\{(-1,\ldots,-1)\}$ be the set of all $2^k-1$ truth value combinations that satisfy a Boolean clause
	(i.e., everything but `all-false').
{Also}, let $\cM(d,k,n)$ be the set of all probability distributions $\mu=(\mu(\sigma))_{\sigma\in\Sigma}$ on $\Sigma$ such that
 $m\mu(\sigma)$ is an integer for all $\sigma\in\Sigma$ and
	\begin{align}\label{eqMuSumsToZero}
	\sum_{\sigma\in\Sigma}\mu(\sigma)\scal{\sigma}{\vecone}=0.
	\end{align}
{(The relevance of this constraint will be made clear.)}  
In addition, define $Z_\mu=Z_\mu(\PHI)$ as the number of truth assignments $\tau:\{x_1,\ldots,x_n\}\to\{\pm1\}$ such that
	\begin{align*}
	\mu(\sigma)&=\frac1m\sum_{i=1}^m\prod_{j=1}^k\vecone\{\sign(i,j)\tau(\partial(i,j))=\sigma_j\}\qquad\mbox{for all }		\sigma\in\Sigma.
	\end{align*}
In words, $Z_\mu$ is the number of satisfying assignments of $\PHI$ such that for each $\sigma\in\Sigma$
precisely $m\mu(\sigma)$ clauses are satisfied according to the ``truth value pattern'' $\sigma$. Since the total number of true literals and false literals are equal,
	all possible distributions on $\Sigma$ {satisfy~(\ref{eqMuSumsToZero}) and} are included in $\cM(d,k,n)$, and thus
	\begin{equation}\label{eqZdecomposition}
	Z=\sum_{\mu\in\cM(d,k,n)}Z_\mu.
	\end{equation}
{We also observe {that} 
the total number of true/false literal occurrences is divisible by $d$ (because (\ref{eqConfModel}) is a bijection).}
{Crucially, (\ref{eqZdecomposition}) decomposes the random variable $Z$, whose value is typically exponential in $n$ for the regime of $d,k$
that we deal with, into a {\em polynomial} number $|\cM(d,k,n)|\leq O(n^{|\Sigma|})$ of summands.}

We are going to apply small subgraph conditioning to the individual random variables $Z_\mu$ rather than $Z$.
The key advantage is that we will be able to evaluate the second moment of $Z_\mu$ almost mechanically by
way of the central limit theorem for random permutations~\cite{Erwin}.

This approach is facilitated by the observation that only a fairly small subset of $\cM(d,k,n)$ contributes to (\ref{eqZdecomposition}) significantly.
In fact, recalling $q$ from (\ref{eqq}), define a probability distribution $\bar\mu$ on $\Sigma$ by letting
	\begin{align}\label{eqbarmu}
	\bar\mu(\sigma)&
		=\frac{(q(1-q))^{k/2}}{1-(1-q)^k}\bcfr{q}{1-q}^{\frac12\sum_{j=1}^k\sigma_j}.
	\end{align}
Further, let $\cM_\omega=\cM_\omega(d,k,n)$ be the set of all $\mu\in\cM(d,k,n)$ such that
$\|\mu-\bar\mu\|_2\leq\omega m^{-1/2}$.
Then our strategy is to show that for any fixed number $\omega>0$
the double limit (\ref{eqVarDecomp2}) with $Z$ replaced by $Z_\mu$ vanishes {uniformly for}
 $\mu\in\cM_\omega$.
In \Sec~\ref{Sec_fm} we calculate the first moments of the random variables $Z_\mu$.

\begin{proposition}\label{Prop_fm}
The first moments satisfy
	\begin{align}\label{eqProp_fm1}
	\Erw[Z]&\sim 
		2^n(1-(1-q)^k)^{m+\frac12}\bc{4q(1-q)}^{-dn-\frac12}=\exp(\Omega(n))
		\qquad\mbox{and}\\
		\lim_{\omega\to\infty}&
		\liminf_{n\to\infty}\sum_{\mu\in\cM_\omega}\frac{\Erw[Z_\mu]}{\Erw[Z]}=1.\label{eqProp_fm2}
	\end{align}
Furthermore, for any $\omega>0$ we have
	\begin{align}\label{eqProp_fm3}
	\limsup_{n\to\infty}\max_{\mu\in\cM_\omega}\abs{\ln\Erw[Z_{\mu}]
		+\ln|\cM_\omega|-\ln\Erw[Z]
		}<\infty.
	\end{align}
\end{proposition}

In addition, we need to work out the covariance of $Z_\mu$ and the cycle counts $C_s$.
As a first step, we study the unconditional distribution of the random variables $C_s$.
For $l\geq1$ and $s=(s_2,\ldots,s_{2l+1})\in\{\pm1\}^{2l}$ define
	\begin{align}\label{eqLambdaDelta}
	\lambda_s&=\frac1{2l}\bcfr{k-1}{2}^l(d(d-1))^{l/2}\bcfr{d-1}{d}^{\frac12\sum_{i=1}^ls_{2i}s_{2i+1}}.
	\end{align}

\begin{proposition}\label{Prop_shortCycles}
Let $S\subset\bigcup_{l\geq1}\{\pm1\}^l$ be a {fixed} finite set of {sign patterns}.
Moreover, let $(c_s)_{s\in S}$ be a fixed family of non-negative integers.
Then
	\begin{align}
	\lim_{n\to\infty}\pr\brk{\forall s\in S:C_s=c_s}&=\prod_{s\in S}\pr\brk{\Po(\lambda_s)=c_s}.
	\end{align}
\end{proposition}

\noindent
Further, for $l\geq1$ and $s=(s_2,\ldots,s_{2l+1})\in\{\pm1\}^{2l}$ let
	\begin{align}\label{eqMatrices}
	M_1&=\frac{1}{q}\begin{pmatrix}
		q^2&(1-q)^2\\
		q^2&q^2
		\end{pmatrix},&
			M_{-1}&=\frac1{1-q}
					\begin{pmatrix}
				(1-q)^2&(1-q)^2\\q^2&(1-q)^2
				\end{pmatrix},&
		\delta_s&=-1+\tr\prod_{i=1}^l M_{s_{2i}s_{2i+1}}.
	\end{align}
Since 
	\begin{align}\label{eqMatricesEigenvectors}
	M_1\bink{1-q}{q}&=M_{-1}\bink{1-q}q=\bink{1-q}q,&
	M_{1}\bink{q-1}q&=-M_{-1}\bink{q-1}q=(2q-1)\bink{q-1}{q},
	\end{align}
we obtain
	\begin{align}\label{eqMatricesEigenvectors2} 
	\delta_s&=(-1)^{\sum_{i=1}^l(1-s_{2i}s_{2i+1})/2}(2q-1)^l.
	\end{align}

\begin{proposition}\label{Prop_shortCyclesConditional}
Let $S\subset\bigcup_{l\geq1}\{\pm1\}^l$ be a finite set, let $(c_s)_{s\in S}$ be a family of non-negative integers and let $\omega>0$.
Then
	\begin{align}\label{eqProp_shortCyclesConditional}
	\lim_{n\to\infty}
	\max_{\mu\in\cM_\omega}\abs{
	\frac{\Erw[Z_{\mu}\vecone\{\forall s\in S:C_{s}=c_{s}\}]}{\Erw[Z_{\mu}]}-\prod_{s\in S}\pr\brk{\Po((1+\delta_{s})\lambda_{s})=c_{ s}}}=0.
	\end{align}
Moreover, $\delta_s>-1$ for all $s$, $(2d-1)(k-1)(1-4q(1-q))<1$ and
	\begin{equation}\label{eqNicksSumIsFinite}
	\sum_{l\geq1}\sum_{s\in\{\pm1\}^l}\lambda_s\delta_s^2=-\frac12\ln\bc{1-(2d-1)(k-1)(1-4q(1-q))}.
	\end{equation}
\end{proposition}

\noindent
The proofs of \Prop s~\ref{Prop_shortCycles} and~\ref{Prop_shortCyclesConditional}  can be found in \Sec~\ref{Sec_cycles}.
Finally, in \Sec~\ref{Sec_sm} we establish the following bound on the second moments of the $Z_\mu$ .

\begin{proposition}\label{Prop_sm}
For any $\omega>0$ we have
	\begin{align*}
	\limsup_{n\to\infty}\max_{\mu\in\cM_\omega}{\Erw[Z_{\mu}^2]}/{\Erw[Z_{\mu}]^2}
		\leq\bc{1-(2d-1)(k-1)(1-4q(1-q))}^{-1/2}.
	\end{align*}
\end{proposition}

We now derive \Thm~\ref{Thm_SAT} from \Prop s~\ref{Prop_fm}--\ref{Prop_sm}.
Basically, we are going to argue that the variance of the random variables $Z_\mu$ comes almost entirely from the {variation in their expected values conditional upon} $C_s${, as described at~(\ref{eqVarDecomp})}.
Although we do not use any technical statements from those papers directly, the argument {an adaptation of} 
	 conditioning from~\cite{Janson,MRRW,RobinsonWormald}
{to the present context, which has} one critical twist: instead of working with a single random variable $Z$,
we need to control all the random variables $Z_\mu$ with $\mu\in\cM_\omega$ for a fixed $\omega>0$ simultaneously.
In fact, ultimately we are going to have to take the limit $\omega\to\infty$ as well.
Recalling that $\cF_\ell$ is the $\sigma$-algebra generated by the random variables $C_s$ with $s\in\bigcup_{l\leq\ell}\{\pm1\}^{2l}$,
we begin with the following bound.

\begin{lemma}\label{Thm_sscFamily}
For any $\omega>0$ we have
	\begin{align*}
	\lim_{\ell\to\infty}\limsup_{n\to\infty}\max_{\mu\in\cM_\omega}\Erw\brk{\frac{\Erw[Z_\mu^2|\cF_{\ell}]-\Erw[Z_\mu|\cF_{\ell}]^2}{\Erw[Z_\mu]^2}}=0.
	\end{align*}
\end{lemma}

\begin{proof}
Spelled out in detail, we aim to prove that
	\begin{align*}
	\forall\eps>0\,\exists \ell_0=\ell_0(\eps)>0\,\forall\ell>\ell_0\,\exists n_0=n_0(\eps,\ell)>0\,\forall n>n_0,\mu\in\cM_\omega:
	\Erw\brk{\Erw[Z_\mu^2|\cF_{\ell}]-\Erw[Z_\mu|\cF_{\ell}]^2}< \eps\Erw[Z_\mu]^2.
	\end{align*}
For $\ell\geq1$ and $B>0$ let $\Gamma(\ell,B)$ be the set of all families $c=(c_s)_{s\in\bigcup_{l\leq\ell}\{\pm1\}^{2l}}$ of integers $0\leq c_s\leq B$.
By \Prop s~\ref{Prop_shortCycles} and~\ref{Prop_shortCyclesConditional}
	for any $\eps>0$ we can choose $B=B(\eps)>0$, $\ell_0(\eps)>0$ large enough such that
	for any $\ell\geq\ell_0(\eps)$ for large enough $n\geq n_0(\eps,\ell,B)$ all $\mu\in\cM_\omega$ satisfy {the following (the first is by definition):} 
	\begin{align}
	\Erw[\Erw[Z_\mu|\cF_{\ell}]^2]&
		\geq\sum_{c\in\Gamma(\ell,B)}	
			\prod_{l=1}^\ell\prod_{s\in\{\pm1\}^{2l}}
				\frac{\Erw[Z_\mu{\vecone\{\forall l\leq\ell,s\in\{\pm1\}^{2l}:C_s=c_s\}}]^2}{\pr\brk{\forall l\leq\ell,s\in\{\pm1\}^{2l}:C_s=c_s}}\nonumber\\
			&\geq\exp(-\eps^2)\Erw[Z_\mu]^2
				\sum_{c\in\Gamma(\ell,B)}		\prod_{l=1}^\ell\prod_{s\in\{\pm1\}^{2l}}
				\frac{\pr[\Po((1+\delta_s)\lambda_s)=c_s]^2}{\pr[\Po(\lambda_s)=c_s]}\nonumber\\
			&=\exp(-\eps^2)\Erw[Z_\mu]^2\sum_{c\in\Gamma(\ell,B)}\prod_{l=1}^\ell\prod_{s\in\{\pm1\}^{2l}}
				\frac{((1+\delta_s)\lambda_s)^{2c_s}}{c_s!\lambda_s^{c_s}\exp({2(1+\delta_s)\lambda_s-\lambda_s})}\nonumber\\
			&{=\exp(-\eps^2)\Erw[Z_\mu]^2\prod_{l=1}^\ell\prod_{s\in\{\pm1\}^{2l}}
				\exp(-2(1+\delta_s)\lambda_s+\lambda_s)\sum_{j=0}^{B}\frac{ (1+\delta_s)^{2j}\lambda_s ^{j}}{j!}
				}\nonumber\\
		&\geq\Erw[Z_\mu]^2\exp\brk{-2\eps^2+\sum_{l\geq1}\sum_{s\in\{\pm1\}^{2l}}\delta_s^2\lambda_s}.
	\label{eqThm_sscFamily666}
	\end{align}
{The last step here uses the fact that the number of possible  $\lambda_s$, as defined in~(\ref{eqLambdaDelta}), is   bounded for fixed $k$, $d$ and $l$.
	} As
	$\Erw[Z_\mu^2]=\Erw[\Erw[Z_\mu^2|\cF_{\ell}]]
		=\Erw[\Erw[Z_\mu^2|\cF_{\ell}]-\Erw[Z_\mu|\cF_{\ell}]^2]
			+\Erw[\Erw[Z_\mu|\cF_{\ell}]^2],$
\Prop~\ref{Prop_sm} and (\ref{eqThm_sscFamily666}) imply that for large enough $\ell,n$ and all $\mu\in\cM_\omega$ we have
	\begin{align*}
	\Erw\brk{\Erw[Z_\mu^2|\cF_{\ell,n}]-\Erw[Z_\mu|\cF_{\ell,n}]^2}\leq\eps\Erw[Z_\mu]^2,
	\end{align*}
as desired.
\end{proof}

\begin{corollary}\label{Lemma_ourApproximation}
For any $\alpha>0$ we have
	$\lim_{\ell\to\infty}\limsup_{n\to\infty}\pr\brk{|Z-\Erw[Z|\cF_{\ell}]|>\alpha \Erw[Z]}=0.$
\end{corollary}
\begin{proof}
\Prop~\ref{Prop_fm} shows that for any $\alpha>0$ there is $\omega>0$ such that
	\begin{align}\label{eqChoiceOfOmega}
	\liminf_{n\to\infty}\sum_{\mu\in\cM_\omega}\frac{\Erw[Z_\mu]}{\Erw[Z]}>1-\alpha^2.
	\end{align}
Pick a small $\eps=\eps(\alpha,\omega)$.
By \Lem~\ref{Thm_sscFamily} we can choose $\ell=\ell(\alpha,\eps,\omega)$  large enough
such that for large $n$ all $\mu\in \cM_\omega$ satisfy
	\begin{align}\label{eqCondVarBound}
	\Erw\brk{\Erw[Z_\mu^2|\cF_{\ell}]-\Erw[Z_\mu|\cF_{\ell}]^2}<\eps{\Erw[Z_\mu]^2}.
	\end{align}
Now  define
	\begin{align*}
	X_\mu&=|Z_\mu- \Erw[Z_\mu|\cF_\ell]|\vecone\{|Z_\mu- \Erw[Z_\mu|\cF_\ell]|>\alpha\Erw[Z_\mu]\},&
		X&=\sum_{\mu\in\cM_\omega}X_\mu.
	\end{align*}
Then
	\begin{align}\label{eqXNotTooSmall}
	X<\alpha\sum_{\mu\in\cM_\omega}\Erw[Z_\mu]&\Rightarrow\abs{\sum_{\mu\in\cM_\omega}Z_\mu-\Erw[Z_\mu|\cF_\ell]}
		\leq 2\alpha\sum_{\mu\in\cM_\omega}\Erw[Z_\mu].
	\end{align}
Furthermore, 
using Chebyshev's inequality at the step introducing the variance, 
	\begin{align*}
	\Erw[X_\mu|\cF_\ell]
	&\leq
	\sum_{j\geq0}2^{j+1}\alpha\Erw[Z_\mu]\pr\brk{X_\mu>2^j\alpha\Erw[Z_\mu]}\\
	&\leq
	\sum_{j\geq0}2^{j+1}\alpha\Erw[Z_\mu]\pr\brk{|Z_\mu-\Erw[Z_\mu|\cF_\ell]|>2^j\alpha\Erw[Z_\mu]}\\
	&\leq	
	\sum_{j\geq0} \frac{ \Var[Z_\mu|\cF_\ell]}{2^{j-1}\alpha\Erw[Z_\mu]}
	 \leq\frac{4\Var[Z_\mu|\cF_\ell]}{\alpha\Erw[Z_\mu]}.
	\end{align*}
Hence,
	\begin{align}\label{eqErwX}
	\Erw[X|\cF_\ell]&\leq\frac4\alpha\sum_{\mu\in\cM_\omega}\frac{\Var[Z_\mu|\cF_\ell]}{\Erw[Z_\mu]}
		=\frac4\alpha\Erw[Z]\sum_{\mu\in\cM_\omega}\frac{\Var[Z_\mu|\cF_\ell]}{\Erw[Z_\mu]^2}\frac{\Erw[Z_\mu]}{\Erw[Z]}.
	\end{align}
Further, by \Prop~\ref{Prop_fm}
	there is a number 
	{$\gamma=\gamma(\omega)$} such that $\Erw[Z_\mu]/\Erw[Z]\leq\gamma/|\cM_\omega|$ for all $\mu\in\cM_\omega$.
Therefore, (\ref{eqErwX}) yields
	\begin{align*}
	\Erw[X|\cF_\ell]&\leq
		\frac{4\gamma\Erw[Z]}{\alpha|\cM_\omega|}\sum_{\mu\in\cM_\omega}\frac{\Var[Z_\mu|\cF_\ell]}{\Erw[Z_\mu]^2}.
	\end{align*}
Choosing $\eps$ small enough, we obtain from 
	(\ref{eqCondVarBound}) and {the tower rule} that
	\begin{align}\label{eqNaiveVariance}
	\Erw[X]=\Erw[\Erw[X|\cF_\ell]]&\leq
		\frac{4\gamma\Erw[Z]}{\alpha|\cM_\omega|}\sum_{\mu\in\cM_\omega}\frac{\Erw[\Var[Z_\mu|\cF_\ell]]}{\Erw[Z_\mu]^2}
			\leq\frac{4\eps\gamma\Erw[Z]}{\alpha}\leq\alpha^2\Erw[Z].
	\end{align}
{Combining with (\ref{eqChoiceOfOmega}) and (\ref{eqXNotTooSmall}), for $n$ sufficiently large we obtain}
	\begin{align} 
	\pr\brk{\abs{\sum_{\mu\in\cM_\omega}Z_\mu-\Erw[Z_\mu|\cF_\ell]}
		 \leq 2\alpha\sum_{\mu\in\cM_\omega}\Erw[Z_\mu]} 
		& \geq \pr\brk{X< \alpha\sum_{\mu\in\cM_\omega}\Erw[Z_\mu]}\nonumber \\
		&\geq \pr\brk{X< \alpha(1-2\alpha^2)\Erw[Z]}\nonumber  \\
		&\geq1-2\alpha \nonumber
	\end{align}
{for $\alpha$ sufficiently small (using Markov's inequality and noting that $X$ is non-negative),} as desired.
\end{proof}

\begin{lemma}\label{Lemma_condMean}
Let
	\begin{align}\label{eqLemma_condMean1}
	U_\ell&=\sum_{l=1}^\ell\sum_{s\in\{\pm1\}^{2l}}C_s\ln(1+\delta_s)-\lambda_s\delta_s.
	\end{align}
Then 
	\begin{align}\label{eqLemma_condMean2}
	\limsup_{\ell\to\infty}\limsup_{n\to\infty}\pr\brk{|\ln\Erw[Z|\cF_\ell]-\ln\Erw[Z]-U_\ell|>\eps}=0\qquad\mbox{for any }\eps>0.
	\end{align}
\end{lemma}
\begin{proof}
{Let} $B>0$, let $\cC_B$ be the event that $C_s\leq B$ for all $l\leq\ell$ and $s\in\{\pm1\}^{2l}$ and define
	$U_{\ell,B}=U_\ell\vecone\{\PHI\in\cC_B\}$.
\Prop~\ref{Prop_shortCycles} ensures that for any $\ell,\eps>0$ there is $B>0$ such that 
	\begin{align}\label{eqLemma_condMean_2}
	\pr\brk{\cC_B}&>1-\eps.
	\end{align}
Additionally, choose $\omega>0$ large enough so that for a small enough $\alpha=\alpha(\eps,\ell,B)$ we have
{for $n$ sufficiently large, using~(\ref{eqChoiceOfOmega}), that}	$\sum_{\mu\in\cM_\omega}\Erw[Z_\mu]\geq(1-\alpha)\Erw[Z]$.
Then{, noting $\lambda_s\ge0$ and using~(\ref{eqZdecomposition}),} \Prop s~\ref{Prop_shortCycles} and~\ref{Prop_shortCyclesConditional} imply that
for any {assignment of values to $c_s$, $s\in\{\pm1\}^{2l}$,} with $c_s\leq B$ for all $s$ we have for large  $n$
	\begin{align}
	\Erw[Z|\forall l\leq\ell,s\in\{\pm1\}^{2l}:C_s=c_s]&\geq
		\sum_{\mu\in\cM_\omega}\Erw[Z_\mu|\forall l\leq\ell,s\in\{\pm1\}^{2l}:C_s=c_s]\nonumber\\
		&\geq\exp(-\eps)\Erw[Z_\mu]\prod_{l\leq\ell}\prod_{s\in\{\pm1\}^{2l}}\frac{\pr\brk{\Po((1+\delta_s)\lambda_s)=c_s}}{\pr\brk{\Po(\lambda_s)=c_s}}
					\nonumber\\
		&=\exp(-\eps)\Erw[Z_\mu]\prod_{l\leq\ell,s}(1+\delta_s)^{c_s}\exp(-\delta_s\lambda_s).
			\label{eqLemma_condMean_3}
	\end{align}
Similarly, assuming that $\alpha$ is chosen sufficiently small, for large enough $n$ we have {(bounding $\Erw[Z|W]$ by $\Erw[Z]/\pr[W]$ in the first step)}
	\begin{align}
	\Erw[Z|\forall l\leq\ell,s\in\{\pm1\}^{2l}:C_s=c_s]&\leq\frac{2\alpha\Erw[Z]}{\prod_{l\leq\ell,s}\pr\brk{\Po(\lambda_s)=c_s}}
		+\sum_{\mu\in\cM_\omega}\Erw[Z_\mu|\forall l\leq\ell,s\in\{\pm1\}^{2l}:C_s=c_s]\nonumber\\
		&\leq\exp(\eps)\Erw[Z_\mu]\prod_{l\leq\ell,s}(1+\delta_s)^{c_s}\exp(-\delta_s\lambda_s).			\label{eqLemma_condMean_4}
	\end{align}
Combining (\ref{eqLemma_condMean_2}), (\ref{eqLemma_condMean_3}) and~(\ref{eqLemma_condMean_4}) and taking logarithms completes the proof of 
(\ref{eqLemma_condMean2}).
\end{proof}

\begin{proof}[Proof of \Thm~\ref{Thm_SAT}]
Let $(\Lambda_s)_{l,s}$ be a family of independent Poisson variables with $\Erw\Lambda_s=\lambda_s$.
For $\ell\geq1$ we define
	\begin{align*}
	W_\ell=\prod_{l=1}^\ell\prod_{s\in\{\pm1\}^{2l}}(1+\delta_s)^{\Lambda_s}\exp(-\lambda_s\delta_s).
	\end{align*}
Then \Prop~\ref{Prop_shortCycles} shows that for each $\ell$ the random variables $U_\ell$ from \Lem~\ref{Lemma_condMean} converge in distribution to 
$\ln W_\ell$ as $n\to\infty$.
Moreover, comparing (\ref{eqLambdaDelta}) and (\ref{eqMatricesEigenvectors2})  with (\ref{eqLambdaDeltaThm}),
we see that the distribution of $W_{\ell}$ coincides with the distribution of
	$\prod_{l\leq\ell}\prod_{0\leq t\leq l}(1+\delta_s)^{\Lambda_{l,t}}\exp(-\lambda_{l,t}\delta_{l,t}).$
Furthermore, 
{following \cite[\Sec~5]{Janson} we note that the sequence $(W_\ell)_\ell$ is a martingale because
	$\Erw[(1+\delta_s)^{\Lambda_s}\exp(-\lambda_s\delta_s)]=1$ for all sign patterns $s$ and is in fact $L^2$-bounded as
	 $\Erw[((1+\delta_s)^{\Lambda_s}\exp(-\lambda_s\delta_s))^2]=\exp(\delta_s^2\lambda_s)$ and
		 $\sum_s\delta_s^2\lambda_s<\infty$.
		 Hence, the $L^2$-version of the martingale convergence theorem
implies that $W$ is well-defined and that the $W_\ell$ converge to $W$ almost surely and in $L^2$ as $\ell\to\infty$.}
Therefore, the assertion follows from \Cor~\ref{Lemma_ourApproximation} and {Lemma}~\ref{Lemma_condMean}.
\end{proof}

\section{The first moment}\label{Sec_fm}

\noindent{\em We continue to assume that $k\geq k_0$ and that $d$ satisfies (\ref{eqAssumptionOnd}).}

\medskip\noindent
In this section we prove \Prop~\ref{Prop_fm}.
We begin by calculating $\Erw[Z]$.
By linearity of expectation this comes down to calculating the probability that a fixed truth assignment $\tau:\{x_1,\ldots,x_n\}\to\{\pm1\}$ is satisfying.
With the notation introduced in \Sec~\ref{Sec_overview}, we thus aim to calculate the probability that
	$\min_{i\in[m]}\max_{j\in[k]}\sign(i,j)\tau(\partial(i,j))=1$.
Hence, we need to get a handle on the random $\pm1$-sequence $(\sign(i,j)\tau(\partial(i,j)))_{i\in[m],j\in[k]}$.
Clearly, because every literal has an equal number of positive and negative occurrences, for {\em every} assignment $\tau$ we have
	\begin{equation}\label{eqModelB}
	\sum_{i\in[m],j\in[k]}\sign(i,j)\tau(\partial(i,j))=0.
	\end{equation}

To compute $\Erw[Z]$ we merely specialise the first moment computation that was done in~\cite{SAT} in greater
generality to the regular $k$-SAT model.%
	\footnote{Although it is not included in~\cite{SAT}  explicitly,
		Konstantinos Panagiotou and the first author actually had the proof of \Lem~\ref{Lemma_SB} on the blackboard.
		The formula given for the first moment in~\cite{Rathi} {is}  equivalent but of a slightly different form.}
Thus, following \cite{SAT} we study the sequence $(\sign(i,j)\tau(\partial(i,j)))_{i,j}$ by means of another random
$\pm1$-vector $\vec\chi=(\chi_{ij})_{i\in[m],j\in[k]}$.
With $q$ from (\ref{eqq}) the entries $\chi_{ij}$ are mutually independent such that 
	$\pr[\chi_{ij}=1]=q$ and $\pr[\chi_{ij}=-1]=1-q$.
Consider the event
	$\cB=\cbc{\sum_{i\in[m],j\in[k]}\chi_{ij}=0}.$
Then the following is immediate from (\ref{eqModelB}) and the definition of the random formula $\PHI$.

\begin{fact}\label{Fact_fm}
Let $\tau:\{x_1,\ldots,x_n\}\to\{\pm1\}$ be a truth assignment.
Then the
conditional distribution of $\vec\chi$ given $\cB$ coincides with the distribution of $(\sign(i,j)\tau(\partial(i,j)))_{i,j}$.
\end{fact}

\noindent
Hence, to calculate $\Erw[Z]$ we need to figure out the probability of
		$\cS\textstyle=\cbc{\min_{i\in[m]}\max_{j\in[k]}\chi_{ij}=1}$
given $\cB$.	

\begin{lemma}\label{Lemma_SB}
We have
	$\pr\brk{\cS|\cB}\sim(1-(1-q)^k)^{m+\frac12}\bc{4q(1-q)}^{-dn-\frac12}.$
\end{lemma}

\noindent
We prove \Lem~\ref{Lemma_SB} by calculating $\pr[\cS],\pr[\cB]$ and $\pr[\cB|\cS]$ and applying Bayes' rule.

\begin{claim}\label{Claim_S}
We have $\pr\brk{\cS}=(1-(1-q)^k)^m$.
\end{claim}
\begin{proof}
The probability that for some $i\in[m]$ we have $\max_{j\in[k]}\chi_{ij}=-1$ equals $(1-q)^k$.
Hence, the claim is immediate from the independence of the entries of $\vec\chi$.
\end{proof}

\begin{claim}\label{Claim_B}
We have $\pr\brk{\cB}=\bink{km}{dn}q^{dn}(1-q)^{dn}$.
\end{claim}
\begin{proof}
As $2dn=km$ the assertion follows from the independence of the entries of $\vec\chi$.
\end{proof}

\begin{claim}\label{Claim_BS}
We have $\pr\brk{\cB|\cS}\sim\sqrt\frac{1-(1-q)^k}{2\pi km q(1-q)}\enspace.$
\end{claim}
\begin{proof}
Let $X=\sum_{i=1}^m\sum_{j=1}^k\vecone\{\chi_{ij}=1\}$.
Then $\cB=\{X=0\}$.
Moreover, the choice (\ref{eqq}) of $q$ ensures that
	\begin{equation}\label{eqClaim_BS1}
	\Erw[X|\cS]=\frac{kmq}{1-(1-q)^k}=dn.
	\end{equation}
Further, given $\cS$, $X$ is merely the sum of the independent random variables $X_i=\sum_{j=1}^k\vecone\{\chi_{ij}=1\}$ and
	\begin{align*}
	\Var[X_i|\cS]&=\Var(\Bin_{\geq1}(k,q))
		=\frac{\Var(\Bin(k,q))}{1-(1-q)^k}=\frac{kq(1-q)}{1-(1-q)^k}.
	\end{align*}
Consequently, $\Var(X)=km\cdot\frac{q(1-q)}{1-(1-q)^k}$.
Thus, the assertion follows from (\ref{eqClaim_BS1}) and the local limit theorem for sums of independent random variables~\cite{DavisMcDonald}.
\end{proof}

\begin{proof}[Proof of \Lem~\ref{Lemma_SB}]
By Bayes' rule, Claims~\ref{Claim_S}--\ref{Claim_BS} and Stirling's formula,
	\begin{align*}
	\pr\brk{\cS|\cB}&\sim
		\frac{\pr\brk{\cB|\cS}\pr\brk{\cS}}{\pr\brk\cB}=
		\sqrt\frac{1-(1-q)^k}{2\pi km q(1-q)}\frac{(1-(1-q)^k)^m}{\bink{km}{km/2}(q(1-q))^{km/2}}
		\sim(1-(1-q)^k)^{m+\frac12}\bc{4q(1-q)}^{-dn-\frac12},
	\end{align*}
as claimed.
\end{proof}

Proceeding to the expectation of $Z_\mu$, we let $M(\sigma)$
be the number of indices $i\in[m]$ such that {the random vector $\chi$ satisfies} $\chi_{ij}=\sigma_j$ for all $j\in[k]$ for $\sigma\in\Sigma=\{\pm1\}^k\setminus\{(-1,\ldots,-1)\}$.
Further, for $\mu\in\cM$ let 
	\begin{equation}\label{eqSmu}
	\cS_\mu=\{M(\sigma)=m\mu(\sigma)\mbox{ for all $\sigma\in\Sigma$}\}.
	\end{equation}

\begin{claim}\label{Lemma_Smu}
For any $\mu\in\cM$ we have $\pr\brk{\cS_\mu|\cB\cap\cS}=\bink{m}{m\mu}{(q(1-q))^{dn}}/{\pr\brk{\cB\cap\cS}}$.
\end{claim}
\begin{proof}
The definition of the set $\cM$ ensures that $\cS_\mu\subset\cS\cap\cB$.
Therefore, the lemma follows from the independence of the entries $\chi_{ij}$.
\end{proof}

\begin{proof}[Proof of \Prop~\ref{Prop_fm}]
Combining Fact~\ref{Fact_fm}, \Lem~\ref{Lemma_SB} and multiplying by the total number of truth assignments, we obtain
	\begin{align}\label{eqFirstMomentFormula1}
	\Erw[Z]&\sim2^n(1-(1-q)^k)^{m+\frac12}\bc{4q(1-q)}^{dn+\frac12}.
	\end{align}
Further, expanding (\ref{eqq}), we see that the unique solution $q\in(0,1)$ satisfies
	\begin{align}\label{eqFirstMomentFormula2}
	q&=\frac12-2^{-1-k}+O(k/4^k).
	\end{align}
Hence, recalling (\ref{eqAssumptionOnd}), we obtain
	\begin{align}
	\ln\Erw[Z]&=n\ln2+\bc{2dn/k+1/2}\ln(1-(1-q)^k)-\bc{dn+1/2}\ln(q(1-q)/4)\nonumber\\
		&=n\brk{\ln2+2k^{-1}d\ln(1-(1-q)^k)-d\ln(4q(1-q))}+O(1)
		=4n(2^{-k}+O(k^24^{-k}))=\Omega(n).\label{eqFirstMomentFormula3}
	\end{align}
Finally, (\ref{eqProp_fm1}) follows from (\ref{eqFirstMomentFormula1}) and (\ref{eqFirstMomentFormula3}).

To complete the proof of \Prop~\ref{Prop_fm}, fix a number $\omega>0$.
The definition vectors $\mu\in\cM_\omega$ must satisfy the two conditions
	$\sum_{\sigma\in\Sigma}\mu(\sigma)=1$ and $\sum_{\sigma\in\Sigma}\mu(\sigma)\scal{\vecone}\sigma=0$.
Therefore,
	\begin{equation}\label{eqProp_fm_1}
	|\cM_\omega|=\Theta(m^{-1+|\Sigma|/2}),
	\end{equation}
with the number hidden in the $\Theta\bc\cdot$ dependent on $\omega$, of course.
Further,
Claims~\ref{Claim_S}, \ref{Claim_BS} and Claim~\ref{Lemma_Smu} and Stirling's formula imply that uniformly for all $\mu\in\cM_\omega$,
	\begin{align*}
	\pr\brk{\cS_\mu|\cB\cap\cS}&=\Theta(\sqrt m)\bink{m}{m\mu}\frac{(q(1-q))^{dn}}{(1-(1-q)^k)^m}
		=\Theta(m^{1-|\Sigma|/2})\frac{(q(1-q))^{dn}}{(1-(1-q)^k)^m}\prod_{\sigma\in\Sigma}\mu(\sigma)^{-m\mu(\sigma)}.
	\end{align*}
Rewriting the last expression in terms of the distribution $\bar\mu$ from (\ref{eqbarmu}), we obtain
	\begin{align}\label{eqProp_fm_2}
	\pr\brk{\cS_\mu|\cB\cap\cS}&
		=\Theta(m^{1-|\Sigma|/2})\exp\bc{-m\KL{\mu}{\bar\mu}}\qquad\mbox{uniformly for }\mu\in\cM_\omega.
	\end{align}
Since the Kullback-Leibler divergence, 
whose definition we recall from (\ref{eqKL}),
 attains its global minimum at the point $\mu=\bar\mu$ and because
its second and third derivative are bounded at this point, (\ref{eqProp_fm3}) follows from (\ref{eqProp_fm_1}),
	(\ref{eqProp_fm_2}) and  Fact~\ref{Fact_fm}.
Finally, (\ref{eqProp_fm2}) follows from (\ref{eqProp_fm_2}) because the Kullback-Leibler divergence is strictly convex.
\end{proof}

\section{Counting cycles}\label{Sec_cycles}

\subsection{Proof of \Prop~\ref{Prop_shortCycles}}
{Similar results were proved for bipartite graphs in the second author's PhD thesis~\cite{Wth}. (See Proposition 3.5 for example, and Theorem 3.12 more explicitly for biregular bipartite graphs of large girth.) The (minor) difference here is that the sign patterns of the cycles are specified.}
The key step of the proof is to establish the following lemma.

\begin{lemma}\label{Lemma_cycleMoments}
Let $S\subset\bigcup_{l\geq1}\{\pm1\}^{2l}$ be a finite set and let $(c_s)_{s\in S}$ be a non-negative integer vector.
Then
	\begin{align*}
	\lim_{n\to\infty}\Erw\prod_{s\in S}C_s^{\underline{c_s}}=\prod_{s\in S}\lambda_s^{c_s}.
	\end{align*}
\end{lemma}

\noindent
\Prop~\ref{Prop_shortCycles} is immediate from \Lem~\ref{Lemma_cycleMoments} and standard results on convergence
to the Poisson distribution (e.g., \cite[\Thm~1.23]{Bollobas}).
To prove \Lem~\ref{Lemma_cycleMoments} we recall that the random factor graph $G(\PHI)$ 
is obtained 
{by linking clones of clauses and literals according to the random bijection (\ref{eqConfModel})}.

\begin{claim}\label{Fact_unicyclic}
Fix an integer $b>1$.
	The expected number of sets of at most $b$ vertices that span more than $b$ edges in $G(\PHI)$ is $O(1/n)$.
\end{claim}
\begin{proof}
Suppose that $b_1,b_2>0$ are integers such that $b_1+b_2=b$ and let $b_3>b$.
Let $Y(b_1,b_2,b_3)$ be the number of pairs $(A,B)$ such that $A\subset\{x_1,\ldots,x_n\}$, $|A|=b_1$,
$B\subset[m]$ such that $A\cup B$ spans at least $b_3$ edges in $G(\PHI)$.
Then
	\begin{align}\label{eqFact_unicyclic}
	Y(b_1,b_2,b_3)&\leq\bink{n}{b_1}\bink{m}{b_2}\bink{2db_1}{b_3}\bink{kb_2}{b_3}b_3!(2dn-b_3)!/(2dn)!\enspace;
	\end{align}
indeed, the binomial coefficients count the number of ways of choosing $b_1$ variables, $b_2$ clauses and
$b_3$ ``clones'' of the chosen variables and clauses.
Then there are $b!$ ways of matching these chosen clones up and $(2dn-b_3)!$ ways of joining the remaining clones.
By comparison, the total number of bijections (\ref{eqConfModel}) equals $(2dn)!$.
The r.h.s.\ of (\ref{eqFact_unicyclic}) is $O(1/n)$ because $b_3>b$.
Finally, the assertion follows by summing over all $b_1,b_2$ such that $b_1+b_2=b$ and all $b_3$ such that $b<b_3\leq\min\{2db_1,kb_2\}$.
\end{proof}

Let $l\geq1$ and let $s\in\{\pm1\}^{2l}$.
As a warm-up we calculate $\Erw[C_s]$; in the process we introduce a bit of notation that will prove useful in \Sec~\ref{Sec_shortCyclesConditional} as well.
Each cycle with sign pattern $s$ arises as follows.
We start from some clause vertex $i$ of $G(\PHI)$.
Then we alternate between variable nodes and clause nodes such that the signs decorating the edges that we walk through are as prescribed by $s$.
Finally, the $l$th variable loops back to the original clause that we started from.
Of course, given the starting clause $i$, each such walk can be encoded by specifying the clones of the clause/literal clones that we follow at each step.
Thus, we let $I(s)$ be the set of all families $(j_h,g_h)_{h=2,\ldots,2l+1}$
with $j_h\in[k]$, $g_h\in[d]$ such that
	\begin{itemize}
	\item $j_2\neq j_{2l+1}$ and $j_{2h+1}\neq j_{2h+2}$ for all $h<l$,
	\item $g_{2h}\neq g_{2h+1}$ for all $h\in[l]$ such that $s_{2h}s_{2h+1}=1$.
	\end{itemize}
Then
	\begin{equation}\label{eqCountingCycles1}
	|I(s)|=(k(k-1))^{2l}d^l\prod_{h=1}^{l}(d-1)^{(1+s_{2h}s_{2h+1})/2}d^{(1-s_{2h}s_{2h+1})/2}.
	\end{equation}
Furthermore, for $i\in[m]$ let $\cC_{\PHI}(s,i,j,g)$ be the event that the cycle prescribed by $(j,g)\in I(s)$ materialises from the starting clause $i$.
That is, if we define $i_2=i$ and $i_{2t-1}=i_{2t}=\partial(\partial(i_{2t-2},j_{2t-2}),g_{2t-1})$ for $t\geq2$, then $(i,j)$ satisfies the conditions
	{\bf CY1--CY5} and $\PHI[i_{t},j_{t}]\in\{x_1,\ldots,x_n\}\times\{g_{t}\}\times\{\pm1\}$ for all $t=2,\ldots,2l+1$.

We claim that
	\begin{align}\label{eqErwCs}
	\Erw[C_s]&=\sum_{i=1}^m\sum_{(j,g)\in I(s)}\pr[\cC_{\PHI}(s,i,j,g)]\sim\frac{|I(s)|}{2l}(2kd)^{-l}=\lambda_s.
	\end{align}
Indeed, because $\PHI$ comes from the random bijection (\ref{eqConfModel}),
the probability that for $h\in[2l]$ the $j_{2h}$th clone of clause $i_{2h}$ is connected to the $g_{2h}$th clone with sign $s_{2h}$
of some variable is $(2d)^{-1}+o(1)$.
Further, the probability that the $g_{2h}$th clone with sign $s_{2h+1}$ of this variable is connected to the $j_{2h+1}$th clone
of some clause is $k^{-1}+o(1)$.
Ultimately, the probability that the $g_{2l}$th clone of sign $s_{2l+1}$ of the last variable visited is connected to the $j_{1}$th clone of the starting clause is $(1+o(1))(km)^{-1}$.
Finally, the factor $1/2l$ in (\ref{eqErwCs}) comes from {\bf CY4} and the convention that we consider the clauses with the least index the
starting point of the cycle.

\begin{proof}[Proof of \Lem~\ref{Lemma_cycleMoments}]
It is straightforward to extend the argument from the previous paragraph to the joint factorial moments of the random variables $C_s$.
Hence,  let $S\subset\bigcup_l\{\pm1\}^{2l}$ be finite and let $c=(c_s)_{s\in S}$ be an integer vector with $c_s>0$ for all $s$.
Then $\prod_{s\in S}C_s^{\underline{c_s}}$ is the number of families that contain precisely $c_s$ distinct cycles of type $s$ for each $s\in S$.
By Fact~\ref{Fact_unicyclic} we just need to count families of vertex-disjoint cycles.
To this end, we choose distinct starting clauses $(i(s,g))_{s\in S,g\in[c_s]}$.
Because $\sum_{s\in S}c_s$ remains fixed as $n\to\infty$, the number of choices is $(1+O(1/m))m^{\sum_{s\in S}c_s}$.
Further, for each $s\in S$ and $g\in[c_s]$ we pick $(j_h(s,g),g_h(s,g))_h\in I(s)$.
By the same reasoning as in the calculation of $\Erw[C_s]$, for each $s\in S$, $g\in[c_s]$ the probability that
the desired cycle materialises is $(1+o(1))(2kd)^{-l}m^{-1}$.
In fact, these events are asymptotically independent because we only consider vertex-disjoint cycles and $\sum_{s\in S}c_s=O(1)$ as $n\to\infty$.
Hence \Lem~\ref{Lemma_cycleMoments} follows.
\end{proof}

\subsection{Proof of \Prop~\ref{Prop_shortCyclesConditional}}\label{Sec_shortCyclesConditional}
With respect to  \Prop~\ref{Prop_shortCyclesConditional}, we use the random vector $\vec\chi$ and the other notation from \Sec~\ref{Sec_fm}.
Consider the following experiment for constructing a formula $\hat\PHI$ together with an assignment $\hat\SIGMA$,
which we call the {\em planted distribution}; similar constructions have been used previously~\cite{victor,kSAT,SAT}.
\begin{description}
\item[PL1] choose a truth assignment $\hat\SIGMA:\{x_1,\ldots,x_n\}\to\{\pm1\}$ uniformly at random.
\item[PL2] choose $\vec\chi$ independently of $\hat\SIGMA$ given that $\vec\chi\in\cS\cap\cB$.
\item[PL3] choose bijection $\hat\PHI:[m]\times[k]\to\{x_1,\ldots,x_n\}\times[d]\times\{\pm1\}$ uniformly subject to the condition
		$$\sign_{\hat\PHI}(i,j)\hat\SIGMA(\partial_{\hat\PHI}(i,j))=\chi_{ij}\qquad\mbox{for all }(i,j)\in[m]\times[k].$$

\end{description}

\noindent
In words, we first choose a truth assignment $\hat\SIGMA$ uniformly at random.
Then, we prescribe a random sequence $\vec\chi$ of $km$ truth values subject to the condition that for each clause index $i\in[m]$ there exists $j\in[k]$
such that $\chi_{ij}=1$ and such that $\sum_{i,j}\chi_{ij}=0$.
Finally, we randomly match those literal occurrences that the assignment $\hat\SIGMA$ renders true to the
precisely the $dn$ clause slots $(i,j)$ such that $\chi_{ij}=1$
and the ones that $\hat\SIGMA$ sets to false to the $dn$ remaining positions.
As an immediate consequence of Fact~\ref{Fact_fm} we obtain

\begin{fact}\label{Fact_planting}
Let $\cA$ be a set of pairs $(\Phi,\sigma)$ of formulas and assignments.
Moreover, let $Z_{\cA}(\PHI)$ be the number of satisfying assignments $\sigma$ of $\PHI$ such that $(\PHI,\sigma)\in\cA$.
Then
	$\Erw[Z_{\cA}(\PHI)]=\Erw[Z(\PHI)]\cdot\pr[(\hat\PHI,\hat\SIGMA)\in\cA]$.
\end{fact}

We are going to use \Lem~\ref{Lemma_muCycles}  to prove the following statement.
Let $I(s)$ be as in the previous section.
As before we are going to be interested in the event that for a clause index $i$ and $(j,g)\in I(s)$ the event $\cC_{\hat\PHI}(i,j,g,s)$
that a cycle as described by $i,j,g,s$ occurs in the formula $\hat\PHI$.
Further, for $i\in[m]$ let $\cC_{\hat\PHI}(i,s)$ be the event that there exists $(j,g)\in I(s)$ such that $\cC_{\hat\PHI}(i,j,g,s)$ occurs.

\begin{lemma}\label{Lemma_plantedCycles}
Let $S\subset\bigcup_{l\geq1}\{\pm1\}^{2l}$ be finite and let $c=(c_s)_{s\in S}$ be a non-negative integer vector.
Let $\vec i=(\vec i(s,a))_{s\in S,a\in[c_s]}$ be a random vector whose entries $\vec i(s,a)\in[m]$ are independent and uniformly distributed.
Then
	\begin{align*}
	\pr\brk{\bigcap_{s\in S,a\in[c_s]}\cC_{\hat\PHI}(\vec i(s,a),s)
		}&\sim\prod_{s\in S}\bcfr{(1+\delta_s)\lambda_s}{m}^{c_s}.
	\end{align*}
\end{lemma}

\noindent
The proof of \Lem~\ref{Lemma_plantedCycles} is based on the following elementary observation.

\begin{claim}\label{Claim_justAFewClauses}
Let $\cI\subset[m]$ be a set of size $|\cI|\leq n^{1/3}$ and let $\tau=(\tau_i)_{i\in\cI}\in\Sigma^{\cI}$.
Further, let $\cE(\cI,\tau)$ be the event that for each $i\in\cI$ we have $(\hat\SIGMA(\hat\PHI[i,j]))_{j\in[k]}=\tau_i$.
Then
	\begin{align}\label{eqClaim_justAFewClauses}
	\pr\brk{\cE(\cI,\tau)}&\sim\prod_{i\in\cI}\frac{((1-q)q)^{k/2}\prod_{j=1}^k (q/(1-q))^{\tau_{ij}/2}}{2q}.
	\end{align}
\end{claim}
\begin{proof}
Since $1-(1-q)^k=2q$ by the definition of $q$, the r.h.s.\ of (\ref{Claim_justAFewClauses}) is just the probability
that $\vec\chi_{ij}=\tau_{ij}$ for all $i\in\cI$, $j\in[k]$ given the event $\cS$.
Moreover, because $|\cI|\leq n^{1/3}$ a similar application of the local limit theorem as in the proof of Claim~\ref{Claim_BS}
shows that $\pr\brk{\cB|\cS}\sim\pr\brk{\cB|\cS,\cE(\cI,\tau)}$.
Therefore, the assertion follows from Bayes' rule.
\end{proof}

\begin{proof}[Proof of \Lem~\ref{Lemma_plantedCycles}]
A similar calculation as in the proof of Claim~\ref{Fact_unicyclic} shows that we only need to consider families of vertex-disjoint cycles.
Further, because the total number of vertices involved in the cycles remains bounded as $n\to\infty$, the events $\cC(\vec i(s,a))$ are asymptotically independent.
Therefore, we are just going to calculate the probability of a single event $\cC(\vec i,s)$ for a random $\vec i\in[m]$.

We can write $\cC(\vec i,s)$ as a disjoint union of sub-events in which we specify the truth values that $\hat\SIGMA$ assigns to the
literals in the order in which they appear along the cycle.
Thus, let $\xi=(\xi_2,\ldots,\xi_{2l+1})\in\{\pm1\}^{2l}$ be a sequence such that $\xi_{2h}\xi_{2h+1}=s_{2h}s_{2h+1}$ for all $h$.
Further, set $\xi_1=\xi_{2l+1}$, $j_1=j_{2l+1}$.
Moreover, let $i=(i_1,\ldots,i_l)\in[m]^l$ be a sequence of pairwise distinct clause indices and let $(j,g)\in I(s)$.
Let $\cD_h(i,j,g,\xi,s)$ be the event that
	$\vec\chi_{i_h,j_{2h-1}}=\xi_{2h+1}$ and $\vec\chi_{i_h,j_{2h}}=\xi_{2h+2}$ and let $\cD(i,j,g,\xi,s)=\bigcap_{h=1}^{l}\cD_h(i,j,g,\xi,s)$.
By symmetry, the probability $$M'_{s_{2h}s_{2h+1}}(\xi_{2h},\xi_{2(h+1)})=\pr[\cD_h(i,j,g,\xi,s)]$$
	depends on $s_{2h}s_{2h+1}$ and $\xi_{2h},\xi_{2(h+1)}$
only.
In fact, using Claim~\ref{Claim_justAFewClauses} we can work out the probabilities  of the eight possible cases easily.
	\begin{description}
	\item[Case 1: $s_{2h}s_{2h+1}=1$]
		there are four sub-cases depending on the truth values $\xi_{2h},\xi_{2(h+1)}$.
		\begin{description}
		\item[Case 1a: $\xi_{2h}=\xi_{2(h+1)}=1$] clause $i_h$ is satisfied because $\xi_{2h}=1$.
			Therefore,  by 
			the probability that $\xi_{2(h+1)}=\xi_{2h+1}=1$ comes to $M'_{1}(1,1)\sim q^2/(2q)=q/2$.
		\item[Case 1b: $\xi_{2h}=-\xi_{2(h+1)}=1$]
			clause $i_h$ is satisfied due to $\xi_{2h}=1$.
			Hence, $M'_{1}(1,-1)\sim
				(1-q)^2/(2q)$.
		\item[Case 1c: $\xi_{2h}=-\xi_{2(h+1)}=-1$]
			clause $i_h$ is satisfied as $\xi_{2h+1}=\xi_{2(h+1)}=1$.
			Hence,
				$M'_1(-1,1)\sim q^2/(2q)=q/2.$
		\item[Case 1d: $\xi_{2h}=\xi_{2(h+1)}=-1$]
			one of the $k-2$ remaining literals of clause $i_h$ has to take the value 1 to satisfy the clause.
			Since $(1-q)^k=1-2q$ and thus $(1-q)^{k-2}=(1-2q)/(1-q)^2$,  we get
				$$M'_1(-1,-1)\sim(1-q)^2(1-(1-q)^{k-2})/(2q)=q/2.$$
		\end{description}
	\item[Case 2: $s_{2h}s_{2h+1}=-1$]
		once more there are four sub-cases.
		\begin{description}
		\item[Case 2a: $\xi_{2h}=\xi_{2(h+1)}=1$] clause $i_h$ is satisfied because $\xi_{2h}=1$.
			Therefore, $M'_{-1}(1,1)\sim q(1-q)/(2q)=(1-q)/2$.
		\item[Case 2b: $\xi_{2h}=-\xi_{2(h+1)}=1$]
			clause $i_h$ is satisfied due to $\xi_{2h}=1$.
			Hence, $M'_{-1}(1,-1)\sim q(1-q)/(2q)=(1-q)/2$.
		\item[Case 2c: $\xi_{2h}=-\xi_{2(h+1)}=-1$]
			for clause $i_h$ to be satisfied one of the $k-2$ literals in the clause that do not belong to the cycle has to be true.
			Thus,
				$$M'_{-1}(-1,1)\sim \frac{q(1-q)(1-(1-q)^{k-2})}{2q}=\frac{q((1-q)^2-(1-q)^k)}{2q(1-q)}=\frac{q^2}{2(1-q)}.$$
		\item[Case 2d: $\xi_{2h}=\xi_{2(h+1)}=-1$]
			clause $i_h$ is satisfied as $\xi_{2h+1}=1$.
			Therefore,
				$M'_{-1}(-1,-1)\sim q(1-q)/(2q)=(1-q)/2.$
		\end{description}		
	\end{description}
Moreover, the by Claim~\ref{eqClaim_justAFewClauses} the events $(\cD_h(i,j,g,\xi,s))_h$ are asymptotically independent.
Therefore, taking the union over all possible truth values $\xi$, we obtain (following similar arguments in~\cite{Janson})
	\begin{align*}
	\pr\brk{\bigcup_\xi\cD(i,j,g,\xi,s)}&
		\sim\sum_\xi\prod_{h=1}^{l}M'_{s_{2h}s_{2h+1}}(\xi_{2h},\xi_{2(h+1)})=
			\tr\prod_{i=1}^lM_{s_{2i}s_{2i+1}}'.
	\end{align*}

Further, with $M_{\pm1}$ the matrices {from (\ref{eqMatrices})}, we see that $M'_{\pm1}\sim\frac12M_{\pm1}$.
Hence, as truth values of the $2l$ literals on the cycle determine the truth values of the $l$ literals on the cycle, we obtain
	\begin{align*}
	\pr\brk{\cC_{\hat\PHI}(\vec i,s)}&\sim2^{l}\pr\brk{\cC_{\PHI}(\vec i,s)}\tr\prod_{i=1}^lM'_{s_{2i}s_{2i+1}}
		=\pr\brk{\PHI\in\cC(\vec i,s)}\tr\prod_{i=1}^lM_{s_{2i}s_{2i+1}}.
	\end{align*}
Thus, the assertion follows from \Lem~\ref{Lemma_cycleMoments}.
\end{proof}

\noindent
While \Lem~\ref{Lemma_plantedCycles} puts us in a position to calculate the covariance of $Z$ and the cycle counts $C_s$,
\Prop~\ref{Prop_shortCyclesConditional} deals with the covariance of $Z_\mu$ and the $C_s$.
Hence, we need to consider a variant of the planted distribution that fixes the clause marginal $\mu$.
We recall the event $\cS_\mu$ from (\ref{eqSmu}).

\begin{lemma}\label{Lemma_muCycles}
For any $\omega>0$ the following is true.
Let $S\subset\bigcup_{l\geq1}\{\pm1\}^{2l}$ be finite and let $c=(c_s)_{s\in S}$ be a non-negative integer vector.
Let $\vec i=(\vec i(s,a))_{s\in S,a\in[c_s]}$ be a random vector whose entries $i(s,a)\in[m]$ are independent and uniformly distributed.
Further, let $\cC=\bigcap_{s\in S,a\in[c_s]}\cC_{\hat\PHI}(\vec i(s,a),s)$.
Then
	\begin{align*}
	\pr\brk{\cS_\mu|\cS\cap\cC}\sim\pr\brk{\cS_\mu|\cS}\qquad\mbox{uniformly for all }\mu\in\cM_\omega.
	\end{align*}
\end{lemma}
\begin{proof}
Suppose that the event $\cS\cap\cC$ occurs.
Let $J\subset[m]$ be the set of all indices of clauses that participate in the cycles corresponding to $\cC$.
Then $m'=|J|=m-O(1)$.
Further, for each $\sigma\in\Sigma$ let $m\mu''(\sigma)$ be the number of clauses $i\in [m]\setminus J$ that are satisfied according to $\sigma\in\Sigma$.
Additionally, let $m\mu'(\sigma)$ be such that $m(\mu'(\sigma)+\mu''(\sigma))=\mu(\sigma)$.
Finally, let $\cS_\mu'$ be the event that the empirical distribution of patterns on the clauses $\bar J=[m]\setminus J$ works out to be precisely $\mu'$.
Then given $\cC$ the event $\cS_\mu$ occurs iff $\cS_\mu'$ occurs.
Hence, in analogy to Claim~\ref{Lemma_Smu} we have
	\begin{align}\nonumber
		\pr\brk{\cS_\mu|\cS\cap\cC}&=
		\bink{m'}{m'\mu'}\frac{(q(1-q))^{dn}}{\pr\brk{\cS}}\sim
			(2\pi m)^{1-2^{k-1}}\brk{\prod_{\sigma\in\Sigma}\mu'(\sigma)}^{-1/2}\exp\brk{-|\bar J|\KL{\mu'}{\bar\mu}}\\
		&\sim\frac{(2\pi m)^{1-2^{k-1}}}{\prod_{\sigma\in\Sigma}\sqrt{\bar\mu(\sigma)}}
			\exp\brk{-(m-O(1))\KL{\mu'}{\bar\mu}}\qquad\mbox{uniformly for }\mu\in\cM_\omega.
				\label{eqLemma_muCycles1}
	\end{align}
Further, since $|J|=O(1)$ we have $\|\mu-\mu'\|_2=O(1/m)$, whence $\KL{\mu'}{\bar\mu}-\KL{\mu}{\bar\mu}=o(m^{-1})$.
Moreover, as $\|\mu-\bar\mu\|_2=O(m^{-1/2})$, we have $\KL{\mu}{\bar\mu}=O(1/m)$.
Therefore, (\ref{eqLemma_muCycles1}) implies that uniformly for $\mu\in\cM_\omega$,
	\begin{align}\nonumber
		\pr\brk{\cS_\mu|\cS\cap\cC}&\sim
			\frac{(2\pi m)^{1-2^{k-1}}}{\prod_{\sigma\in\Sigma}\sqrt{\bar\mu(\sigma)}}
			\exp\brk{-m\KL{\mu}{\bar\mu}}\sim\pr\brk{\cS_\mu|\cS},
	\end{align}
as claimed.
\end{proof}

\noindent
Combining \Lem s~\ref{Lemma_plantedCycles} and~\ref{Lemma_muCycles}, we obtain

\begin{corollary}\label{Cor_plantedCycles}
Let $\omega>0$.
With the notation from \Lem~\ref{Lemma_plantedCycles} we have
	\begin{align*}
	\pr\brk{\bigcap_{s\in S,a\in[c_s]}\cC_{\hat\PHI}(\vec i(s,a),s)\bigg|\cS_\mu
		}&\sim\prod_{s\in S}\bcfr{(1+\delta_s)\lambda_s}{m}^{c_s}
		\qquad\mbox{uniformly for all }\mu\in\cM_\omega.
	\end{align*}
\end{corollary}

\noindent
Now, (\ref{eqProp_shortCyclesConditional}) follows from Fact~\ref{Fact_planting}, \Cor~\ref{Cor_plantedCycles} and
the standard result on convergence to the Poisson distribution (e.g., \cite[\Thm~1.23]{Bollobas}).
Hence, we are left to prove

\begin{lemma}
The series $\sum_{l,s}\delta_s^2\lambda_s$ converges and satisfies (\ref{eqNicksSumIsFinite}).
\end{lemma}
\begin{proof}
Being the solution to (\ref{eqq}), $q$ satisfies $q=\frac12+O(2^{-k})$.
Hence, our assumption that $d\leq k2^k$ ensures that
	\begin{align*}
	|(2d-1)(k-1)(1-4q(1-q))|<1.
	\end{align*}
{Therefore, merely plugging in the expressions from (\ref{eqLambdaDelta}) and (\ref{eqMatricesEigenvectors2}), 
we obtain
	\begin{align*}
	\sum_{l\geq1}\sum_{s\in\{\pm1\}^{2l}} \delta_s^2\lambda_s=\sum_{l\geq1}\frac1{2l}\brk{(2d-1)(k-1)(1-4q(1-q))}^l
		=-\frac12\ln\bc{1-(2d-1)(k-1)(1-4q(1-q))}<\infty,
	\end{align*}
as claimed.}
\end{proof}

\section{The second moment}\label{Sec_sm}

\subsection{Outline}
In this section we prove \Prop~\ref{Prop_sm}.
Let $Z_{\alpha,\mu}^\tensor$ be the number of pairs $(\tau_1,\tau_2)$ of satisfying assignments
such that $\mu=\mu(\PHI,\tau_1)=\mu(\PHI,\tau_2)$ and such that $\sum_{i=1}^n\vecone\{\tau_1(x_i)=\tau_2(x_i)\}=\alpha$.
Then by the linearity of expectation
	\begin{align}\label{eqsmmDist}
	\Erw[Z_\mu^2]&=\sum_{\alpha=0}^n\Erw[Z_{\mu,\alpha}^\tensor].
	\end{align}
We {will} evaluate the sum on the r.h.s.\ of (\ref{eqsmmDist}) in two steps.
The main step is to calculate the contribution of $\alpha$ close to $n/2$.

\begin{lemma}\label{Lemma_centre}
Uniformly for $\mu\in\cM_\omega$ we have
	\begin{align*}
	\lim_{a\to\infty}\lim_{n\to\infty}\sum_{\alpha:|\alpha-\frac n2|\leq a\sqrt n}\frac{\Erw[Z_{\mu,\alpha}^\tensor]}{\Erw[Z_\mu]^2}
		=\bc{1-(2d-1)(k-1)(1-4q(1-q))}^{-1/2}.
	\end{align*}
\end{lemma}

\noindent
The proof of \Lem~\ref{Lemma_centre} is based on the following {lemma}, which we derive from the central limit theorem
for random permutations~\cite{Erwin} in \Sec~\ref{Sec_permutations}. {The motivation for the definition of $\vec y^\tensor$ in this lemma will become clear very soon, in the proof of \Lem~\ref{Lemma_centre}.}

\begin{lemma}\label{Lemma_perm}
The following holds uniformly for all $\mu\in\cM_\omega$.
Let $$\vec y^\tensor=(y_{ij}^{(1)},y_{ij}^{(2)})_{(i,j)\in[m]\times[k]}$$ be chosen uniformly at random from the set of all {$m\times k$ 
$\{\pm1\}^2$-arrays}.
Let $\cS_\mu^\tensor$ be the event that
	\begin{align}\label{eqLemma_perm1}
	\sum_{i=1}^m\prod_{j=1}^k\vecone\{y_{ij}^{(1)}=\sigma_j\}=\sum_{i=1}^m\prod_{j=1}^k\vecone\{y_{ij}^{(2)}=\sigma_j\}=m\mu(\sigma)
		\qquad\mbox{for all $\sigma\in\Sigma$}.
	\end{align}
Further, let 
	\begin{align*}
	A&=\sum_{i=1}^m\sum_{j=1}^k\vecone\{y_{ij}^{(1)}=y_{ij}^{(2)}=1\},&\nu^2=\frac{k}{16}\bc{k-4(k-1)q(1-q)}.
	\end{align*}
Then uniformly for all reals $a<b$ we have
	\begin{align*}
	\pr[m^{-1/2}(A-dn/2)\in(a,b)|\cS_\mu^\tensor]&=\frac1{\sqrt{2\pi}\nu}\int_a^b\exp(-z^2/(2\nu^2))dz+o(1).
	\end{align*}
\end{lemma}

\begin{proof}[Proof of \Lem~\ref{Lemma_centre}.]
Fix $\omega>0$ and let $\mu\in\cM_\omega$.
There are $2^n\bink{n}{\alpha}$ pairs $(\tau_1,\tau_2)$ of truth assignments with overlap $\alpha$.
If we fix one such pair $(\tau_1,\tau_2)$, what is the probability that $\mu(\tau_1)=\mu(\tau_2)=\mu$?
To {determine this we need information} on the distribution of the string
	$$(\sign(i,j)\tau^{(1)}(\partial(i,j)),\sign(i,j)\tau^{(2)}(\partial(i,j)))_{i\in[m],j\in[k]}$$
of truth value combinations that emerges if we match the literals to the clauses randomly and plug in truth values according to $\tau^{(1)},\tau^{(2)}$.
Set
	$$\tilde y_{ij}^{(t)}=\sign(i,j)\tau^{(1)}(\partial\PHI[i,j])\qquad(i\in[m],j\in[k],t\in\{1,2\})$$
and $\tilde{\vec y}^\tensor=(\tilde y_{ij}^{(1)},\tilde y_{ij}^{(2)})_{i,j}$ for the sake of brevity.
Moreover, let $\vec y^\tensor$ be uniformly random as in \Lem~\ref{Lemma_perm}.
Further, with the notation of \Lem~\ref{Lemma_perm} {define the two events}
	\begin{align*}
	\cA_\alpha&=\{A=d\alpha\},&
		\cB^\tensor&=\cbc{
		\sum_{i=1}^m\sum_{j=1}^ky_{ij}^{(1)}=\sum_{i=1}^m\sum_{j=1}^ky_{ij}^{(2)}=0}.
	\end{align*}
Then the distribution of $\tilde{\vec y}^\tensor$ coincides with the distribution of $\vec y^\tensor$ given $\cA_\alpha\cap\cB^\tensor$.
Clearly, because $\vec y^\tensor$ is uniformly random, Stirling's formula yields 
	\begin{align*}
	\pr\brk{\vec y^\tensor\in\cA_\alpha|\cB^\tensor}&=
		\bink{2dn}{\alpha d,(n-\alpha)d,(n-\alpha)d,\alpha d}\bink{2dn}{dn}^{-2}=\bink{dn}{d\alpha}^{2}/\bink{2dn}{dn}
		\sim\frac{2}{\sqrt{\pi dn}}\exp\bc{-4dn\bc{\frac\alpha n-\frac12}^2}
	\end{align*}
uniformly for all $\alpha$ such that $|\alpha-n/2|\leq n^{0.6}$ and all $\mu\in\cM_\omega$.
Additionally, we can write
	\begin{align*}
	\pr\brk{\vec{ y}^\tensor\in\cS_\mu^\tensor|\cB^\tensor}&=4^{-n}\Erw[Z_\mu]^2.
	\end{align*}
Further, by Bayes' rule
	\begin{align*}
	\pr\brk{\mu(\tau_1)=\mu(\tau_2)=\mu}&=\pr\brk{\vec{\tilde y}^\tensor\in\cS_\mu^\tensor}
		=\pr\brk{\vec{ y}^\tensor\in\cS_\mu^\tensor|\vec{ y}^\tensor\in\cA_\alpha\cap\cB^\tensor}
		=\frac{\pr\brk{\vec{ y}^\tensor\in\cA_\alpha|\cS_\mu^\tensor}
			\pr\brk{\vec{ y}^\tensor\in\cS_\mu^\tensor|\cB^\tensor}}
			{\pr\brk{\vec{ y}^\tensor\in\cA_\alpha|\cB^\tensor}}\\
		&=\frac{\sqrt{\pi dn}\Erw[Z_\mu]^2}{2^{2n+1}}\exp\bc{4dn\bc{\frac\alpha n-\frac12}^2}
			\pr\brk{\vec{ y}^\tensor\in\cA_\alpha|\cS_\mu^\tensor}.
	\end{align*}
Moreover, uniformly for all $\alpha$ such that $|\alpha-n/2|\leq n^{0.6}$,
	\begin{align*}
	2^{-n}\bink{n}{\alpha n}&\sim\sqrt{\frac2{\pi n}}\exp\brk{-2n\bc{\frac\alpha n-\frac12}^2}.
	\end{align*}
Consequently, \Lem~\ref{Lemma_perm} gives
	\begin{align*}
	\sum_{\alpha:|\alpha-\frac n2|\leq a}\frac{\Erw[Z_{\mu,\alpha}^\tensor]}{\Erw[Z_\mu]^2}&=o(1)+
		\sqrt{\frac{d}{2}}\sum_{\alpha:|\alpha-\frac n2|\leq a\sqrt n}
		\exp\bc{(4d-2)\bc{\frac\alpha n-\frac12}^2n}\pr\brk{\vec{ y}^\tensor\in\cA_\alpha|\cS_\mu^\tensor}\\
		&=o(1)+\sqrt{\frac{d n}{4\pi\nu^2m}}\int_{-\infty}^\infty
			\exp\brk{\bc{4d-2-\frac{4dk}{\nu^2}}\bc{z-\frac12}^2}dz.
	\end{align*}
Taking $a\to\infty$, we thus obtain
	\begin{align*}
	\lim_{a\to\infty}\lim_{n\to\infty}\sum_{\alpha:|\alpha-\frac n2|\leq a}\frac{\Erw[Z_{\mu,\alpha}^\tensor]}{\Erw[Z_\mu]^2}&=
		\sqrt{\frac{k}{8\pi\nu^2}}\int_{-\infty}^\infty
			\exp\brk{\bc{4d-2-\frac{dk}{4\nu^2}}\bc{z-\frac12}^2}dz
		=\sqrt\frac k{2dk-16(2d-1)\nu^2}\enspace.
	\end{align*}
Plugging in the expression for $\nu^2$ and simplifying completes the proof.
\end{proof}

\noindent
Building upon ideas from~\cite{SAT},
in \Sec~\ref{Sec_boundary} we prove the following bound on the contribution of $\alpha$ far from $n/2$.

\begin{lemma}\label{Lemma_sm}
Uniformly for $\mu\in\cM_\omega$ we have
	\begin{align*}
	\lim_{a\to\infty}\lim_{n\to\infty}\sum_{\alpha:|\alpha-\frac n2|>a\sqrt n}\frac{\Erw[Z_{\mu,\alpha}^\tensor]}{\Erw[Z_\mu]^2}=0.
	\end{align*}
\end{lemma}

\noindent
Finally, \Prop~\ref{Prop_sm} is immediate from \Lem s~\ref{Lemma_sm} and~\ref{Lemma_centre}.

\subsection{Proof of \Lem~\ref{Lemma_perm}}\label{Sec_permutations}
We begin by calculating the expectation and the variance of $A$ given $\cS_\mu^\tensor$ {as defined in~(\ref{eqLemma_perm1})}. {This is the number of $1$s in common between the arrays ${\bf y}^{(1)}$ and ${\bf y}^{(2)}$, conditional on both arrays having row frequencies specified by $\mu$.}
To simplify the notation we denote by $\hat{\vec y}=(\hat y_{ij}^{(1)},\hat y_{ij}^{(2)})_{i,j}$ the random vector $\vec y^\tensor$ given that $\cS_\mu^\tensor$ occurs.

\begin{lemma}\label{Lemma_perm_1}
We have $\Erw[A(\hat{\vec y})]=dn/2 +O(1)$ and $\Var(A(\hat{\vec y}))\sim\nu^2m$, where
	\begin{align*}
	\nu^2&=\frac{k}{16}\bc{k-4(k-1)q(1-q)}.
	\end{align*}
\end{lemma}

The proof will show that $\nu^2$ is $\Var(A(\hat{\vec y}))$ in the case that $\mu=\bar\mu$.
\begin{proof}
 Let
	$A_{ij}=\vecone\{\hat y^{(1)}_{ij}=\hat y^{(2)}_{ij}=1\}$ so that
 $A=\sum_{i,j}A_{ij}$.
To calculate the expectation, set 
	$$A_j=\sum_{i\in[m]}\vecone\{\hat y^{(1)}_{ij}=\hat y^{(2)}_{ij}=1\},\quad a_j=\sum_{\sigma\in\Sigma}\vecone\{\sigma_j=1\}\mu(\sigma)
		=\frac1m\sum_{i\in[m]}\vecone\{\hat y^{(1)}_{ij}=1\}=
		\frac1m\sum_{i\in[m]}\vecone\{\hat y^{(2)}_{ij}=1\}.$$
{Thus, $a_j$ is the fraction of clauses whose $j$th literal is `true' in a truth assignment that contributes to $Z_\mu$.}
Then it is clear that $\Erw[A]=\sum_{j=1}^k\Erw[A_j]$ and $\Erw[A_j]=m a_j^2$.
Furthermore, because $\mu\in\cM_\omega$ we have
	\begin{align*}
	a_j-\frac12&=a_j-\sum_{\sigma\in\Sigma}\vecone\{\sigma_j=1\}\bar\mu(\sigma)&[\mbox{by  (\ref{eqq}) and (\ref{eqbarmu})}]\\
		&=\sum_{\sigma\in\Sigma}\vecone\{\sigma_j=1\}(\mu(\sigma)-\bar\mu(\sigma))\leq2^k\|\mu-\bar\mu\|_2\leq2^k\omega m^{-1/2}.
	\end{align*}
Finally, since $\sum_ja_j=1/2$ by (\ref{eqMuSumsToZero}), for any fixed $\omega>0$ we have
	\begin{align*}
	\Erw[A]-\frac{dn}2&=m\sum_{j=1}^k\bc{a_j^2-\frac14}=m\sum_{j=1}^k\bc{a_j-\frac12}^2\leq4^k\omega^2=O(1).
	\end{align*} 

Moving on to the variance, we expand $\Erw[A^2]$ to obtain
	\begin{align*}
	\Erw[A^2]&=\sum_{i,s=1}^m\sum_{j,t=1}^k\Erw[A_{ij}A_{st}]
		=\Erw[A]+\sum_{i,j,s,t:i=s,t\neq j}\Erw[A_{ij}A_{st}]+\sum_{i,j,s,t:i\neq s,j=t}\Erw[A_{ij}A_{st}]+\sum_{i,j,s,t:i\neq s,j\neq t}\Erw[A_{ij}A_{st}]\\
		&=\Erw[A]+\sum_{i,j,t:t\neq j}\Erw[A_{ij}A_{it}]+\sum_{i,j,s,t:i\neq s}\Erw[A_{ij}A_{st}]. 
	\end{align*}
Further, for $\sigma,\tau\in\Sigma$ and $j\in[k]$ let $\zeta_j(\sigma,\tau)=\vecone\{\sigma_j=\tau_j=1\}$ and 
	$\zeta(\sigma,\tau)=\sum_{j\in[k]}\zeta_j(\sigma,\tau)=(k+\scal{\sigma}\tau)/2$.
Then
	\begin{align*}
	\sum_{i,j,t:t\neq j}\Erw[A_{ij}A_{it}]&=
		m^{-1}\sum_{\sigma,\tau}\mu(\sigma)\mu(\tau)\sum_{j\neq t}\zeta_j(\sigma,\tau)\zeta_t(\sigma,\tau)
		=m^{-1}\sum_{\sigma,\tau}\mu(\sigma)\mu(\tau)(\zeta(\sigma,\tau)^2-\zeta(\sigma,\tau))\\
		&=-\Erw[A]+m^{-1}\sum_{\sigma,\tau}\mu(\sigma)\mu(\tau)\zeta(\sigma,\tau)^2.
	\end{align*}
Additionally,
	\begin{align*}
	\sum_{i,j,s,t:i\neq s}\Erw[A_{ij}A_{st}]&=\sum_{\sigma,\tau}\sum_{\sigma',\tau'}
		\frac{\mu(\sigma)\mu(\tau) (\mu(\sigma')-\vecone\{\sigma=\sigma'\}) (\mu(\tau')-\vecone\{\tau=\tau'\})}{m(m-1)}\zeta(\sigma,\tau)\zeta(\sigma',\tau')\\
		&=\frac{\brk{\sum_{\sigma,\tau}\mu(\sigma)\mu(\tau)\zeta(\sigma,\tau)}^2-
			2\sum_{\sigma,\tau,\sigma'}\mu(\sigma)\mu(\tau)\mu(\sigma')\zeta(\sigma,\tau)\zeta(\sigma,\sigma')
				+\sum_{\sigma,\tau}\mu(\sigma)\mu(\tau)\zeta(\sigma,\tau)^2}{m(m-1)}\\
		&=\frac{m\Erw[A]^2}{m-1}-\frac2{m(m-1)}\sum_{\sigma,\tau,\sigma'}\mu(\sigma)\mu(\tau)\mu(\sigma')\zeta(\sigma,\tau)\zeta(\sigma,\sigma')
			+\frac1{m(m-1)}\sum_{\sigma,\tau}\mu(\sigma)\mu(\tau)\zeta(\sigma,\tau)^2.
	\end{align*}
Combining the above, we see that uniformly for $\mu\in\cM_\omega$
	\begin{align*}
	\Var(A)&\sim\frac1m\sum_{\sigma,\tau}\bar\mu(\sigma)\bar\mu(\tau)\zeta(\sigma,\tau)^2
		-\frac2{m^2}\sum_{\sigma,\tau,\sigma'}\bar\mu(\sigma)\bar\mu(\tau)\bar\mu(\sigma')\zeta(\sigma,\tau)\zeta(\sigma,\sigma')
			+\frac1m\brk{\sum_{\sigma,\tau}\bar\mu(\sigma)\bar\mu(\tau)\zeta(\sigma,\tau)}^2.
	\end{align*}
Plugging in the definition of $\bar\mu$ and using (\ref{eqq}), we obtain
	\begin{align*}
	m^{-2}\sum_{\sigma,\tau}\bar\mu(\sigma)\bar\mu(\tau)\zeta(\sigma,\tau)&=
		(1-(1-q)^k)^{-2}\Erw[\Bin(k,q^2)]=k/4,\\
	m^{-2}\sum_{\sigma,\tau}\bar\mu(\sigma)\bar\mu(\tau)\zeta(\sigma,\tau)^2&=
		(1-(1-q)^k)^{-2}\Erw[\Bin(k,q^2)^2]
			=\frac{kq^2((k-1)q^2+1)}{(1-(1-q)^k)^{2}}=\frac k4((k-1)q^2+1),\\
	m^{-3}\sum_{\sigma,\tau,\tau'}\mu(\sigma)\mu(\tau)\mu(\tau')\zeta(\sigma,\tau)\zeta(\sigma,\tau')
		&=(1-(1-q)^k)^{-3}\sum_{j=1}^k\bink kjq^j(1-q)^j\bc{\sum_{l=1}^jlq^l(1-q)^{j-l}}^2\\
		&=\frac{kq^3((k-1)q+1)}{(1-(1-q)^k)^3}=\frac k8((k-1)q+1).
	\end{align*}
Hence,
	\begin{align*}
	m^{-1}\Var(A)&\sim\frac{k}{16}\brk{k-4((k-1)q+1)+4((k-1)q^2+1)}=\frac{k}{16}\bc{k-4(k-1)q(1-q)},
	\end{align*}
as claimed.
\end{proof}

\noindent
Finally, \Lem~\ref{Lemma_perm} follows from \Lem~\ref{Lemma_perm_1} and {Bolthausen's} central limit theorem for random permutations
from~\cite{Erwin}; this result can be viewed as an extension of Berry-Esseen inequality to certain dependent random variables,
	and as such provides a uniform estimate for our purposes.  
To be precise, due to our conditioning on the event $\cS_\mu^\tensor$ the distribution of the random vector
$\hat{\vec y}=(\hat y_{ij}^{(1)},\hat y_{ij}^{(2)})_{i,j}$ can be described as follows.
Fix any vector $\tilde{ u}=(\tilde u_{ij})_{i,j}\in\{\pm1\}^{km}$ such that
$\sum_{i=1}^m\vecone\{(u_{i1},\ldots,u_{ik})=\sigma\}=m\mu(\sigma)$ for every $\sigma\in\Sigma$.
Moreover, let $\vec\pi^{(1)},\vec\pi^{(2)}:[m]\to[m]$ be two independent uniformly random permutations and let
$\tilde{\vec u}^{(t)}=(u_{\vec\pi^{(t)}(i),j})_{i,j}$ for $t=1,2$.
In words, $\tilde{\vec u}^{(t)}$ is obtained from $\tilde u$ by permuting the $m$ blocks of length $k$ that represent the individual clauses randomly.
Then $\hat{\vec y}$ has the same distribution as $(u_{\vec\pi^{(1)}(i),j},u_{\vec\pi^{(2)}(i),j})_{i,j}$.
Hence, $A$ is distributed as
	\begin{align*}
	\sum_{i=1}^m\sum_{j=1}^k\vecone\{u_{\vec\pi^{(1)}(i),j}=u_{\vec\pi^{(2)}(i),j}=1\}
		&=\sum_{i=1}^m\sum_{j=1}^k\vecone\{u_{i,j}=u_{\vec\pi^{(2)\,-1}\circ\vec\pi^{(1)}(i),j}=1\},
	\end{align*}
which is precisely the type of random sum for which~\cite{Erwin} establishes convergence to the normal distribution.

\subsection{Proof of \Lem~\ref{Lemma_sm}}\label{Sec_boundary}
We build upon the following result on the {\em total} number of satisfying assignments, which is implicit in prior work~\cite{SAT};
for the sake of completeness we give a self-contained proof in Appendix~\ref{Sec_Kosta}.
Let $Z_\alpha^\tensor$ be the number of pairs $(\sigma,\tau)$ of satisfying assignments of $\PHI$
such that $\sum_{i=1}^n\vecone\{\tau_1(x_i)=\tau_2(x_i)\}=\alpha${.}

\begin{lemma}\label{Lemma_Kosta}
There exists a number $t_0=t_0(k)$ such that for every $t>t_0$ we have
	$$\limsup_{n\to\infty}\sum_{\alpha:|\alpha-n/2|>t n^{1/2}}\Erw[Z_\alpha^\tensor]/\Erw[Z]^2\leq \exp(-t^2/17).$$
Moreover,
	$$\sum_{\rho:|\alpha-n/2|>n^{1/2}\ln n}\Erw[Z_\alpha^\tensor]\leq O(n^{-\ln\ln n})\Erw[Z]^2.$$
\end{lemma}

In the following it will be convenient to replace the parameter $\alpha$ by another overlap parameter to represent the four possible truth value combinations. Define $\rho=(\rho_{s,t})_{s,t=\pm1}$ such that
	\begin{equation}\label{eqOverlapParameter}
	\rho_{1,1}+\rho_{1,-1}=\rho_{1,1}+\rho_{-1,1}=\frac12,\quad\rho_{1,1}+\rho_{1,-1}+\rho_{-1,1}+\rho_{-1,-1}=1.
	\end{equation}
In particular, $\rho$ is a probability distribution on $\{\pm1\}^2$ such that $\rho_{1,1}=\rho_{-1,-1}$ and $\rho_{1,-1}=\rho_{-1,1}$.
Hence,  (\ref{eqOverlapParameter}) {demonstrates} that we can view $\rho_{1,-1},\rho_{-1,1},\rho_{-1,-1}$ as affine functions of $\rho_{1,1}$.
The relationship between $\rho$ and $\alpha$ is going to be
	$2d\alpha=km(\rho_{1,1}+\rho_{-1,-1})=2km\rho_{1,1}$.
Indeed, let us introduce the symbols
	\begin{align}\label{eqrhovsalpha}
	Z_{\rho}^\tensor&=Z_{km\rho_{11}/d}^\tensor,&Z_{\mu,\rho}^\tensor&=Z_{\mu,km\rho_{11}/d}^\tensor.
	\end{align}
 We need to obtain a result similar to Lemma~\ref{Lemma_Kosta} for $Z_{\mu,\rho}^\tensor$ rather than $Z_{\alpha}^\tensor$.  Slightly extending the argument from~\cite{SAT},
we tackle the second moment computation by way of
 an auxiliary probability space as in \Sec~\ref{Sec_fm}.
 To unclutter the notation we write $f(k)=\tilde O(g(k))$ if there exists $c>0$ such that $|f(k)|\leq k^c g(k)$ for all $k>c$. 

\begin{lemma}\label{Lemma_parameterShift}
For any $\rho$ there exists a unique probability distribution $(q_{z_1,z_2})_{z_1,z_2\in\{\pm1\}}$ on $\{\pm1\}^2$ such that
	\begin{align}\label{eqLemma_parameterShifteqn}
	\frac{q_{1,1}}{1-2(q_{-1,-1}+q_{-1,1})^k+q_{-1,-1}^k}&=\rho_{1,1},&
			\frac{q_{1,-1}(1-(q_{-1,-1}+q_{1,-1})^{k-1})}{1-2(q_{-1,-1}+q_{-1,1})^k+q_{-1,-1}^k}&=\rho_{1,-1},&
			q_{-1,1}&=q_{1,-1}.
	\end{align}
The derivatives satisfy
	\begin{align}\label{eqLemma_parameterShift}
	\frac{\partial q_{1,1}}{\partial\rho_{1,1}}=1+\tilde O(2^{-k}),\quad\frac{\partial q_{1,-1}}{\partial\rho_{1,1}}=-1+\tilde O(2^{-k}),\quad
	\frac{\partial^2 q_{1,1}}{\partial\rho_{1,1}^2}=\tilde O(2^{-k}),\quad\frac{\partial^2 q_{1,-1}}{\partial\rho_{1,1}^2}=\tilde O(2^{-k}).
	\end{align}
\end{lemma}
\begin{proof}
Let $\cQ$ be the set of all probability distributions $(q_{\pm1\pm1})$ such that  $q_{1,-1}=q_{-1,1}$.
Further, set
	\begin{align*}
	s&=1-2(q_{-1,-1}+q_{1,-1})^k+q_{-1,-1}^k,&Q_{1,1}&=\frac{q_{1,1}}{s},&
		Q_{1,-1}&=\frac{q_{1,-1}(1-(q_{-1,-1}+q_{1,-1})^{k-1})}{s}.
	\end{align*}
Then we aim to study the function $q\mapsto(Q_{1,1},Q_{1,-1})$ on the 2-dimensional compact convex set $\cQ$.
Since $q_{1,-1}=q_{-1,1}$ we have $q_{1,-1}\leq\frac12$ on $\cQ$. 
Similarly, $q_{1,1}\leq\frac12$ and $q_{1,1}+q_{1,-1}\leq\frac12$.
Consequently, $s=1-O(2^{-k})$ and $Q_{1,1},Q_{1,-1}$ are well-defined.
The derivatives of $s$ work out to be
	\begin{align*}
	\frac{\partial s}{\partial q_{1,1}}&=2k(q_{-1,-1}+q_{1,-1})^{k-1}-kq_{-1,-1}^{k-1},&
		\frac{\partial s}{\partial q_{1,-1}}&=2k(q_{-1,-1}+q_{1,-1})^{k-1}-2kq_{-1,-1}^{k-1}.
	\end{align*}
Further,
	\begin{align*}
	\frac{\partial Q_{1,1}}{\partial q_{1,1}}&=\frac{1}{s}-\frac{q_{1,1}}{s^2}\frac{\partial s}{\partial q_{1,1}},\qquad
		\frac{\partial Q_{1,1}}{\partial q_{1,-1}}=\frac{q_{11}}{s^2}\frac{\partial s}{\partial q_{1,-1}},\\
	\frac{\partial Q_{1,-1}}{\partial q_{1,1}}&=\frac{(k-1) q_{1,-1}(q_{-1,-1}+q_{1,-1})^{k-2}}{s}-
				\frac{Q_{1,-1}}{s^2}\frac{\partial s}{\partial q_{1,1}},\\
		\frac{\partial Q_{1,-1}}{\partial q_{1,-1}}&=\frac{1-(q_{-1,-1}+q_{1,-1})^{k-1}+(k-1) q_{1,-1}(q_{-1,-1}+q_{1,-1})^{k-2}}{s}
			-\frac{Q_{1,-1}}{s^2}\frac{\partial s}{\partial q_{1,-1}}.
	\end{align*}
Since $q_{1,-1}\leq\frac12,q_{1,1}\leq\frac12,q_{1,1}+q_{1,-1}\leq\frac12$ on $\cQ$, we see that
	\begin{align*}
	\frac{\partial Q_{1,1}}{\partial q_{1,1}}&=1+O(k2^{-k}),&
		\frac{\partial Q_{1,1}}{\partial q_{1,-1}}&=O(k2^{-k}),&
	\frac{\partial Q_{1,-1}}{\partial q_{1,1}}&=O(k2^{-k}),&
		\frac{\partial Q_{1,-1}}{\partial q_{1,-1}}&=1+O(k2^{-k}).
	\end{align*}
Consequently, the Jacobi matrix is invertible on $\cQ$.
Further, for any  value $q_{1,-1}\in[0,1/2]$ we have $\lim_{q_{1,1}\to 0}Q_{1,1}=0$ and $\lim_{q_{1,1}\to 1/2}Q_{1,1}>1/2$.
Similarly, for  $q_{1,1}\in[0,1/2]$ we have $\lim_{q_{1,-1}\to 0}Q_{1,-1}=0$ and $\lim_{q_{1,-1}\to 1/2}Q_{1,-1}>1/2$.
Therefore, the assertion follows from the inverse function theorem.
\end{proof}

Define a random vector
	$$\vec\chi^\tensor=\vec\chi^\tensor(q)=(\chi_{ij}^{(1)},\chi_{ij}^{(2)})_{i\in[m],j\in[k]}\quad\mbox{such that}\quad
		\pr\brk{(\chi_{ij}^{(1)},\chi_{ij}^{(2)})=(z_1,z_2)}=q_{z_1,z_2}\quad(z_1,z_2\in\{\pm1\})$$
independently for all $i\in[m]$, $j\in[k]$.
Let 
	$$b_{z_1,z_2}=\frac1{km}\sum_{i,j}\vecone\{\chi_{ij}^{(1)}=z_1,\chi_{ij}^{(2)}=z_2)\}\qquad\mbox{for all }z_1,z_2\in\{\pm1\}.$$
Further, let
	$\cB^\tensor(\rho)$ be the event that $b=\rho$.
In analogy to Fact~\ref{Fact_fm} we have

\begin{fact}\label{Fact_sm}
Let $\tau_1,\tau_2:\{x_1,\ldots,x_n\}\to\{\pm1\}$ be truth assignments with overlap $\rho$.
Then the conditional distribution of $\vec\chi^\tensor$ given $\cB^\tensor(\rho)$ coincides with the distribution of the vector
	$(\sign(i,j)\tau_1(\partial(i,j)),\sign(i,j)\tau_2(\partial(i,j)))_{i\in[m],j\in[k]}$.
\end{fact}

\noindent
Fact~\ref{Fact_sm} as well as the following three claims already appear in~\cite{SAT};
	we include the short proofs for the sake of completeness.

\begin{claim}\label{Claim_Brho}
Uniformly for all $\rho,q$ such that $\rho_{s,t},q_{s,t}\in[1/8,3/8]$ for all $s,t\in\{\pm1\}$ we have
	\begin{align*}
	\ln\pr\brk{\cB^\tensor(\rho)}&=-\frac32\ln n-km\KL{\rho}{q}+O(1).
	\end{align*}
Furthermore, uniformly for all $\rho$ we have
	\begin{align*}
	\ln\pr\brk{\cB^\tensor(\rho)}&\leq-km\KL{\rho}{q}+O(1).
	\end{align*}
\end{claim}
\begin{proof}
We have
	\begin{align*}
	\pr\brk{\cB^\tensor(\rho)}&=\bink{km}{\rho km}\prod_{z_1,z_2\in\{\pm1\}}q_{z_1,z_2}^{km\rho_{z_1,z_2}}.
	\end{align*}
The claim follows by applying Stirling's formula.
\end{proof}

\noindent
Further, consider the event
 $$\cS^\tensor=\cbc{\forall i\in[m]\,\exists j,j'\in[k]:
	\chi_{ij}^{(1)}=\chi_{ij'}^{(1)}=1}.$$
If we think of the $k$-tuples $(\chi_{ij}^{(1)})_{j\in[k]},(\chi_{ij}^{(2)})_{j\in[k]}$ as the truth value combinations induced on a clause
by a pair $(\tau_1,\tau_2)$ of Boolean assignments, then $\cS^\tensor$ corresponds to the event that both $\tau_1,\tau_2$ are satisfying.

\begin{claim}\label{Claim_SsecondMoment}
We have
	$\pr\brk{\cS^\tensor}=(1-2(q_{-1,-1}+q_{-1,1})^k+q_{-1,-1}^k)^m$.
\end{claim}
\begin{proof}
By inclusion{-}exclusion, for any $i\in[m]$ we have
	\begin{align*}
	\pr\brk{\exists h\in\{1,2\}:\forall j\in[k]:\chi_{ij}^{(h)}=-1}&=(q_{-1,-1}+q_{-1,1})^k+(q_{-1,-1}+q_{1,-1})^k-q_{-1,-1}^k
		=2(q_{-1,-1}-q_{-1,1})^k+q_{-1,-1}^k.
	\end{align*}
The assertion follows from the independence of the entries of $\vec\chi^\tensor$.
\end{proof}

\begin{claim}\label{Claim_BsecondMoment}
Uniformly for all $\rho,q$ such that $\rho_{s,t},q_{s,t}\in[1/8,3/8]$ for all $s,t\in\{\pm1\}$ we have
	\begin{align*}
	\ln\pr\brk{\cB^\tensor(\rho)|\cS^\tensor}&=-\frac32\ln n+O(1).
	\end{align*}
\end{claim}
\begin{proof}
This follows from the local limit theorem for sums of independent bounded random variables (e.g.,~\cite{DavisMcDonald}).
\end{proof}

Departing from the argument in~\cite{SAT}, we are now going to accommodate the additional constraint that
the clause marginals follow some specific distribution $\mu$ on $\Sigma$.
Hence, let $\cM_m(\rho)$ be the set of all probability distributions $\vec\nu=(\nu(\sigma,\tau))_{\sigma,\tau\in\Sigma}$ such that
$m\nu(\sigma,\tau)$ is an integer for all $\sigma,\tau\in\Sigma$ and
	\begin{align*}
	\sum_{i=1}^k\sum_{\sigma,\tau\in\Sigma}\nu(\sigma,\tau)\vecone\{\sigma_i=s,\tau_i=t\}=\rho_{s,t}\qquad\mbox{for all }s,t\in\{\pm1\}.
	\end{align*}
Additionally, for a given probability distribution $\mu=(\mu(\sigma))_{\sigma\in\Sigma}$ let $\cM_m(\rho,\mu)$ be the set of all $\vec\nu\in\cM_m(\rho)$ such that
	\begin{align*}
	\sum_{\tau\in\Sigma}\nu(\sigma,\tau)=\sum_{\tau\in\Sigma}\nu(\tau,\sigma)&=\mu(\sigma)\qquad\mbox{for all }\sigma\in\Sigma.
	\end{align*}
Clearly, the vector $\vec\chi^\tensor$ induces a distribution $\vec\nu_{\vec\chi^\tensor}$ by
	\begin{align*}
	\nu_{\vec\chi^\tensor}(\sigma,\tau)&=\frac1m\sum_{i=1}^m\prod_{j=1}^k\vecone\{\chi_{ij}^{(1)}=\sigma_i,\chi_{ij}^{(2)}=\tau_i\}.
	\end{align*}
Letting  $p(\vec\nu)=\pr\brk{\vec\nu_{\vec\chi^\tensor}=\vec\nu|\cB^\tensor(\rho)\cap\cS^\tensor}$ for $\vec\nu\in\cM_m(\rho)$,
	 recalling (\ref{eqrhovsalpha}) and using Fact~\ref{Fact_sm}, we find
	\begin{align*}
	\Erw[Z_{\mu,\rho}^\tensor]&=\Erw[Z_\rho^\tensor]\sum_{\vec\nu\in\cM_m(\rho,\mu)}p(\vec\nu).
	\end{align*}
Let $\bar{\vec\nu}(\rho)=(\bar\nu_{\sigma,\tau}(\rho))_{\sigma,\tau\in\Sigma}$ with 
	\begin{equation}\label{eqbarnu}
	\bar\nu_{\sigma,\tau}(\rho)=\frac1s\prod_{i=1}^kq_{\sigma(i),\tau(i)}.
	\end{equation}
{Then Fact~\ref{Fact_sm} shows that $\bar{\vec\nu}(\rho)$ describes the expected statistics of the ``clause overlaps'' given overlap $\rho$.
More precisely, if we fix two truth assignments with overlap $\rho$ and then generate a random formula subject to the condition
that both assignments are satisfying, then we expect to see
$\bar\nu_{\sigma,\tau}(\rho)m$ clauses that are satisfied according to the truth value pattern $\sigma$ under the first assignment
and according to the truth value pattern $\tau$ under the second one.}
By Stirling's formula, 
	\begin{align}\label{eqpnu}
	p(\vec\nu)&=\frac1{\pr\brk{\cB^\tensor(\rho)|\cS^\tensor}}\bink{m}{m\vec\nu}\prod_{s,t\in\{\pm1\}}q_{s,t}^{km\rho_{s,t}}
		=\frac{r(\vec\nu)}{\pr\brk{\cB^\tensor(\rho)|\cS^\tensor}}\exp(-m\KL{\vec\nu}{\bar{\vec\nu}(\rho)}),\qquad\mbox{where}\\
	r(\vec\nu)&\sim(2\pi m)^{(1-|\Sigma|^2)/2}
		\prod_{\sigma,\tau\in\Sigma}\bar\nu_{\sigma,\tau}(\rho)^{-\frac12}\qquad\mbox{uniformly for $\nu$ s.t.\ }
			|\nu_{\sigma,\tau}-\bar\nu_{\sigma,\tau}|\leq m/\ln m\mbox{ for all }\sigma,\tau
				\label{eqpnu2}
	\end{align}
and $r(\nu)=O(1)$ for all $\vec\nu$.

\begin{claim}\label{Claim_cutoff1}
If $|\rho_{1,1}-\frac14|\leq\ln n/\sqrt n$, then $|\bar\nu_{\sigma,\tau}(\bar\rho)-\nu_{\sigma,\tau}(\rho)|\leq\ln^2n/\sqrt n$.
\end{claim}
\begin{proof}
This follows from (\ref{eqbarnu}) and the fact that the derivatives of the implicit parameter $q$ are bounded.
\end{proof}

\begin{claim}\label{Claim_cutoff2}
Uniformly for $\rho$ such that $|\rho_{1,1}-\frac14|\leq\ln n/\sqrt n$ we have
	\begin{align*}
	\Erw[Z_{\mu,\rho}^\tensor]&\sim\Erw[Z_{\mu,\rho}^\tensor]
	\sum_{\nu\in\cM_m(\rho):\|\vec\nu-\bar{\vec\nu}(\rho)\|_\infty\leq m^{-1/3}}p(\vec\nu).
	\end{align*}
\end{claim}
\begin{proof}
This follows from (\ref{eqpnu}) and the fact that the Kullback-Leibler divergence is strictly convex.
\end{proof}

\begin{proof}[Proof of \Lem~\ref{Lemma_sm}]
Let $a>0$.
By (\ref{eqProp_fm_2}) we have
	\begin{align}\label{eqProp_fm_2_a}
	\Erw[Z_\mu]/\Erw[Z]&=\Theta(m^{1-|\Sigma|/2})
	\end{align}
uniformly for all $\mu\in\cM_\Omega$.
Therefore, letting
	\begin{align*}
	S&=\sum_{\rho:a<|\rho_{1,1}-\frac14|\leq n^{-\frac12}\ln n}\Erw[Z_{\mu,\rho}^\tensor],
	\end{align*}
we obtain from \Lem~\ref{Lemma_Kosta} and Claim~\ref{Claim_cutoff2} that
	\begin{align*}
	S&\sim S'=
		\sum_{\rho:a<|\rho_{1,1}-\frac14|\leq n^{-\frac12}\ln n}\Erw[Z_\rho^\tensor]
		\sum_{\vec\nu\in\cM_m(\rho,\mu):\|\vec\nu-\bar{\vec\nu}(\rho)\|_\infty\leq m^{-1/3}}p(\vec\nu).
	\end{align*}
Hence, (\ref{eqpnu}) and (\ref{eqpnu2}) yield
	\begin{align*}
	S'&\sim S''=
		\frac{(2\pi m)^{(1-|\Sigma|^2)/2}}{\prod_{\sigma,\tau\in\Sigma}\bar\nu_{\sigma,\tau}(\bar\rho)^{\frac12}}
		\sum_{\rho:a<|\rho_{1,1}-\frac14|\leq n^{-\frac12}\ln n}\frac{\Erw[Z_\rho^\tensor]}{\pr\brk{\cB^\tensor|\cS^\tensor}}
			\sum_{\vec\nu\in\cM_m(\rho,\mu):\|\vec\nu-\bar{\vec\nu}(\rho)\|_\infty\leq m^{-1/3}}
			\exp\brk{-m\KL{\vec\nu}{\bar{\vec\nu}(\rho)}}.
	\end{align*}
Estimating the last sum via the Laplace method and using Claim~\ref{Claim_cutoff1} once more, we see that uniformly for all $\rho,\mu$ (again using convexity of the Kullback-Leibler divergence)
	\begin{align*}
	\sum_{\vec\nu\in\cM_m(\rho,\mu):\|\vec\nu-\bar{\vec\nu}(\rho)\|_\infty\leq m^{-1/3}}\exp\brk{-m\KL{\vec\nu}{\bar{\vec\nu}(\rho)}}&
		\leq O(m^{(|\Sigma|^2-2|\Sigma|)/2}).
	\end{align*}
Consequently, Claim~\ref{Claim_BsecondMoment} yields 
	\begin{align*}
	S''&\leq O(m^{2-|\Sigma|})\Erw[Z]^2\exp(-a^2/16),
	\end{align*}
provided that $a$ is sufficiently large.
Therefore, the assertion follows from (\ref{eqProp_fm_2_a}).
\end{proof}

\begin{appendix}

\section{Proof of \Lem~\ref{Lemma_Kosta}}\label{Sec_Kosta}

\noindent
We continue to use the notation from \Sec~\ref{Sec_boundary}.

\begin{claim}
There exists a number $t_0=t_0(k)$ such that for every $t>t_0$ we have
	$$\sum_{\rho:t_0n^{-1/2}<|\rho_{11}-1/4|<2^{-0.49k}}\Erw[Z_\rho^\tensor(\PHI)]\leq \exp(-t^2/4)\Erw[Z]^2.$$
\end{claim}
\begin{proof}
The proof is based on the Laplace method.
Specifically, let $q=q(\rho)$ be the vector from \Lem~\ref{Lemma_parameterShift}.
Then at the point $\rho=\bar\rho=\frac14\vecone$ we can express the vector $(q_{\pm1,\pm1})$ in terms of the solution $q$ to (\ref{eqq}).
Indeed, letting $q_{1}=q$, $q_{-1}=1-q$, we verify that 
the probability distribution $(q_sq_t)_{s,t=\pm1}$ satisfies (\ref{eqLemma_parameterShifteqn}).
Hence, \Lem~\ref{Lemma_parameterShift} implies that $q_{s,t}=q_sq_t$ for $s,t=\pm1$ at the point $\rho=\bar\rho$.
Plugging this distribution in and using Fact~\ref{Fact_sm}, Claim~\ref{Claim_SsecondMoment}, Claim~\ref{Claim_BsecondMoment}, Stirling's formula and Bayes' rule, we find
	\begin{align}\label{eqsmmf}
	\frac{\Erw[Z_\rho^\tensor(\PHI)]}{\Erw[Z^2]}&\leq O(n^{-1/2})\exp\brk{n(H(\rho)-2\ln 2)+m(f(\rho)-f(\bar\rho))},\qquad\mbox{where}\\
	f(\rho)&=\ln(1-2(q_{-1,-1}+q_{-1,1})^k+q_{-1,-1}^k)
		+k\KL{\rho}{q}.
		\nonumber
	\end{align}
We are going to prove that
	\begin{align}\label{eqsmmFirstDerivative}
	Df(\bar\rho)&=0,\\
	D^2f(\rho)&\preceq\frac n{km}\id\qquad\mbox
		{for all $\rho$ such that }|\rho_{11}-1/4|\leq 2^{-0.49k}.
			\label{eqsmmSecondDerivative}
	\end{align}
Since the entropy satisfies $DH(\bar\rho)=0$ and $D^2H(\rho)\preceq-\id$ if $|\rho_{11}-1/4|\leq 2^{-0.49k}$,
the assertion follows from (\ref{eqsmmf})--(\ref{eqsmmSecondDerivative}) and a Gaussian summation.

To prove (\ref{eqsmmFirstDerivative})\footnote{%
	The following fairly simple way of calculating $Df(\bar\rho)$ was pointed out to the first author by Victor Bapst.}
 we set
	$$f_1(\rho)=\ln(1-2(q_{-1,-1}+q_{-1,1})^k+q_{-1,-1}^k),\qquad f_2(\rho)=k\KL{\rho}{q}.$$
Further, let $s=1-2(q_{-1,-1}+q_{-1,1})^k+q_{-1,-1}^k$.
Then
	\begin{align}\label{eqsmmFirstDerivative1}
	\frac{\partial f_1}{\partial q_{1,1}}&=\frac{2k(q_{-1,-1}+q_{-1,1})^{k-1}-k q_{-1,-1}^{k-1}}{s},&
	\frac{\partial f_1}{\partial q_{1,-1}}&=\frac{2k(q_{-1,-1}+q_{-1,1})^{k-1}-q_{-1,-1}^{k-1}}{s}.
	\end{align}
Moreover, the partial  derivatives of the generic term $z\ln(z/y)$ of the Kullback-Leibler divergence work out to be
	\begin{align}\label{eqsmmSecondDerivative3}
	\frac{\partial}{\partial z}z\ln\frac zy&=\ln\frac zy,&
		\frac{\partial}{\partial y}z\ln\frac zy&=-\frac zy=-1-\frac{z-y}{y}.
	\end{align}
Hence, 
	\begin{align}	\label{eqsmmFirstDerivative2}
	\frac{\partial}{\partial q_{1,1}}\KL{\rho}q&=-\frac{\rho_{1,1}}{q_{1,1}}+\frac{\rho_{-1,-1}}{q_{-1,-1}},&
	\frac{\partial}{\partial q_{1,-1}}\KL{\rho}q&=-\frac{2\rho_{1,-1}}{q_{1,-1}}+\frac{2\rho_{-1,-1}}{q_{-1,-1}}.
	\end{align}
Using the relations $q_{s,t}=q_sq_t$ and $(1-q)^k=1-2q$, at the point $\rho=\bar\rho$ we obtain $s=4q^2$ and
	\begin{align}\label{eqsmmFirstDerivative3}
	\frac{2k(q_{-1,-1}+q_{-1,1})^{k-1}-k q_{-1,-1}^{k-1}}{s}&=\frac{k(1-2q)}{4q^2(1-q)^2},&
	\frac{2k(q_{-1,-1}+q_{-1,1})^{k-1}-q_{-1,-1}^{k-1}}{s}&=\frac{k(1-2q)}{2q(1-q)^2},\\
	-\frac{\rho_{1,1}}{q_{1,1}}+\frac{\rho_{-1,-1}}{q_{-1,-1}}&=\frac1{4(1-q)^2}-\frac1{4q^2},&
	-\frac{2\rho_{1,-1}}{q_{1,-1}}+\frac{2\rho_{-1,-1}}{q_{-1,-1}}&=\frac{1}{2(1-q)^2}-\frac{1}{2q(1-q)}.
		\label{eqsmmFirstDerivative4}
	\end{align}
Plugging (\ref{eqsmmFirstDerivative3})--(\ref{eqsmmFirstDerivative4}) into (\ref{eqsmmFirstDerivative1}) and (\ref{eqsmmFirstDerivative2}) and simplifying,
we obtain
	\begin{align}\label{eqsmmFirstDerivative5}
	\frac{\partial}{\partial q_{1,1}}f(\rho)&=\frac{\partial}{\partial q_{1,-1}}f(\rho)=0.
	\end{align}
Further, combining (\ref{eqsmmSecondDerivative3}) and (\ref{eqsmmFirstDerivative5}) and using the chain rule, we get
	\begin{align}\label{eqsmmFirstDerivative6}
	\frac{\partial}{\partial\rho_{1,1}}f_2(\rho)&=\frac{\partial f(\rho)}{\partial q_{1,1}}\frac{\partial q_{1,1}}{\partial \rho_{1,1}}
		+\frac{\partial f(\rho)}{\partial q_{1,-1}}\frac{\partial q_{1,-1}}{\partial \rho_{1,1}}+
		\ln\frac{1}{4q^2}-2\ln\frac1{4q(1-q)}+\ln\frac{1}{4(1-q)^2}=0.
	\end{align}
Thus, (\ref{eqsmmFirstDerivative}) follows from (\ref{eqsmmFirstDerivative5}), (\ref{eqsmmFirstDerivative6}) and the chain rule.

With respect to the second derivative,  letting
	$$u=\frac{2k(k-1)(q_{-1,-1}+q_{-1,1})^{k-2}s-4k^2(q_{-1,-1}+q_{-1,1})^2}{s^2},$$
we find
	\begin{align}\label{eqsmmSecondDerivative1}
	\frac{\partial f_1}{\partial q_{1,1}},\frac{\partial f_1}{\partial q_{1,-1}}&=\tilde O(2^{-k}),&
	\frac{\partial^2 f_1}{\partial q_{1,\pm1}\partial q_{1,\pm1}}&=u+\tilde O(4^{-k}).
	\end{align}
Combining (\ref{eqsmmSecondDerivative1}) with (\ref{eqLemma_parameterShift}) and using the chain rule, we obtain
	\begin{align}\label{eqsmmSecondDerivative2}
	\frac{\partial^2 f_1}{\partial \rho_{11}^2}&=\tilde O_k(4^{-k}).
	\end{align}
Proceeding to $f_2$, we recall that the second differentials of the genertic term $z\ln(z/y)$ of the Kullback-Leibler divergence read
	\begin{align}
	\frac{\partial^2}{\partial z^2}z\ln\frac zy&=\frac 1z,&
			\frac{\partial^2}{\partial y^2}z\ln\frac zy&=\frac z{y^2},&
			\frac{\partial^2}{\partial y\partial z}z\ln\frac zy&=-\frac1y.
				\label{eqsmmSecondDerivative4}
	\end{align}
We verify that in the case $|\rho_{11}-\frac14|\leq2^{-0.49k}$ the implicit parameters satisfy $q_{\pm1,\pm1}-\rho_{\pm1,\pm1}=\tilde O_k(2^{-k})$.
Moreover, $q_{-1,1}=q_{1,-1}$ and $q_{-1,-1}=1-q_{-1,1}-q_{1,-1}-q_{1,1}$.
Therefore, (\ref{eqsmmSecondDerivative3}) yields
	\begin{align*}
	\frac{\partial f_2}{\partial q_{1,1}},\frac{\partial f_2}{\partial q_{1,-1}}&=\tilde O_k(2^{-k}).
	\end{align*}
Hence, by (\ref{eqLemma_parameterShift}) and the chain rule,
	\begin{align}				\label{eqsmmSecondDerivative5}
	\frac{\partial f_2}{\partial q_{1,1}}\frac{\partial^2 q_{1,1}}{\partial\rho_{11}^2}
		+\frac{\partial f_2}{\partial q_{1,-1}}\frac{\partial^2 q_{1,-1}}{\partial\rho_{11}^2}&=\tilde O_k(4^{-k}).
	\end{align}
Further, using (\ref{eqsmmSecondDerivative4}) and~(\ref{eqLemma_parameterShift}) and recalling that 
$q_{\pm1,\pm1}=\frac14+O(2^{-0.49k})$ if $|\rho_{11}-\frac14|\leq2^{-0.49k}$, we obtain
	\begin{align}\label{eqsmmSecondDerivative6}
	\frac{\partial^2 f_2}{\partial^2 q_{1,1}}\bcfr{\partial q_{1,1}}{\partial\rho_{1,1}}^2+
	\frac{\partial^2 f_2}{\partial^2 q_{1,-1}}\bcfr{\partial q_{1,-1}}{\partial\rho_{1,1}}^2
	+2\frac{\partial^2 f_2}{\partial q_{1,1}\partial q_{1,-1}}
		\frac{\partial q_{1,1}}{\partial\rho_{1,1}}
		\frac{\partial q_{1,-1}}{\partial\rho_{1,1}}=\tilde O_k(4^{-k}).
	\end{align}	
Combining (\ref{eqsmmSecondDerivative5}) and (\ref{eqsmmSecondDerivative6}), we get
	\begin{align}\label{eqsmmSecondDerivative7}
	\frac{\partial^2 f_2}{\partial \rho_{1,1}^2}&=\tilde O_k(4^{-k}).
	\end{align}
Finally, since $m/n=\tilde O_k(2^{-k})$, (\ref{eqsmmSecondDerivative}) follows from (\ref{eqsmmSecondDerivative2}) and (\ref{eqsmmSecondDerivative7}).	
\end{proof}

\begin{claim}
We have
	$$\sum_{\rho:|\rho_{11}-1/4|>2^{-0.49k}}\Erw[Z_\rho^\tensor(\PHI)]\leq \exp(-\Omega(n))\Erw[Z]^2.$$
\end{claim}
\begin{proof}
We observe that Fact~\ref{Fact_sm} holds for {\em any} choice of the auxiliary variables $(q_{\pm1,\pm1})$ that define
the random vector $\vec\chi^\tensor$.
Hence, choosing $q=\rho$ and applying Bayes' rule, we find
	\begin{align*}
	\Erw[Z_\rho^\tensor]&\leq\exp\brk{n\bc{H(\rho)+\frac{2d}k\ln(1-2^{1-k}+\rho_{11}^k)}+o(n)}.
	\end{align*}
We claim that
	$$g(\rho_{1,1})=H(\rho)+\frac{2d}k\ln(1-2^{1-k}+\rho_{1,1}^k)$$
attains its maximum at the boundary point $\rho_{11}=1/4+2^{-0.49k}$.
Indeed, we read off that $g(\rho_{1,1})>g(\frac12-\rho_{1,1})$ if $\rho_{11}>1/4$.
Hence, the maximum occurs in the interval $\rho_{1,1}\in[1/4+2^{-0.49k},1/2)$.
Further, since $g(\rho)$ is a sum of the concave $\rho_{11}\mapsto H(\rho)$ and a multiple of the convex  $\rho_{11}\mapsto\ln(1-2^{1-k}+\rho_{11}^k)$,
it suffices to prove this claim for the maximum value of $2d/k$ that (\ref{eqAssumptionOnd}) allows.
Hence, for this $d$ we need to study the zeros of
	\begin{align*}
	\frac\partial{\partial\rho_{11}}g(\rho)&=2\ln\frac{1-2\rho_{11}}{2\rho_{11}}+\frac{2d\rho_{11}^{k-1}}{1-2^{1-k}+\rho_{11}^k}
	\end{align*}
Setting $x=\frac{1-2\rho_{11}}{2\rho_{11}}\in(0,1)$ and taking exponentials, we transform this problem into finding the solutions to
	\begin{align*}
	x&= \exp\bc{-\frac{2d(1+x)}{(2^k-2)(1+x)^k+1}} 
	\end{align*}
for $x\in(0,\frac12-2^{-0.49k})$.
A bit of elementary calculus shows that there are just two solutions, namely $$x_1=(1+O(k^{-1}))2^{-k}\qquad\mbox{ and }\qquad x_2=\Theta(\ln k/k).$$
The first solution $x_1$ is indeed a local maximum, but a direct calculation yields $g((1+x_1)/2)<g(\frac14+2^{-0.49k})$.
Moreover, $x_2$ is a local minimum.
Finally, the assertion follows from the observation that $g(\frac14+2^{-0.49k})<f(\bar\rho)$.
\end{proof}

\end{appendix}


\begin{thebibliography}{29}


\bibitem{Barriers}
D.~Achlioptas, A.~Coja-Oghlan: 
Algorithmic barriers from phase transitions.
Proc.~49th FOCS (2008) 793--802.

\bibitem{nae}
D.~Achlioptas, C.~Moore:
Random $k$-SAT: two moments suffice to cross a sharp threshold.
SIAM Journal on Computing {\bf 36} (2006) 740--762.

\bibitem{nature}
D.~Achlioptas, A.~Naor, Y.~Peres:
Rigorous location of phase transitions in hard optimization problems.
Nature {\bf 435} (2005) 759--764.

\bibitem{yuval}
D.~Achlioptas, Y.~Peres:
The threshold for random $k$-SAT is $2^k \ln 2 - O(k)$.
Journal of the AMS \textbf{17} (2004) 947--973.


\bibitem{victor}
V.~Bapst, A.~Coja-Oghlan:
The condensation phase transition in the regular $k$-SAT model.
arXiv:1507.03512 (2015).

\bibitem{Bollobas} B.~\Bollobas{:} 
Random graphs.
2nd edition Cambridge (2001).


\bibitem{BBK} A.~B{\'e}k{\'e}ssy, P.~B{\'e}k{\'e}ssy{, J.~Koml\'os:}
Asymptotic
enumeration of regular matrices. {\em Studia Sci.\ Math.\ Hungar.} {\bf 7}
(1972),
343--353.


\bibitem{Erwin}
E. Bolthausen:
An estimate of the remainder in a combinatorial central limit theorem.
Zeitschrift f\"ur Wahrscheinlichkeitstheorie und verwandte Gebiete {\bf 66} (1984) 379--386.


\bibitem{kSAT}
A.~Coja-Oghlan, K.\ Panagiotou:
The asymptotic $k$-SAT threshold.
Advances in Mathematics {\bf288} (2016) 985--1068.


\bibitem{SAT}
A.~Coja-Oghlan, K.~Panagiotou:
Going after the $k$-SAT threshold.
Proc. 45th STOC (2013) 705--714.

\bibitem{Danny}
A.~Coja-Oghlan, D.~Vilenchik:
Chasing the $k$-colorability threshold.
Proc.\ 54th FOCS (2013) 380--389.

\bibitem{Lenka}
A.~Coja-Oghlan, L.~Zdeborov\'a:
The condensation transition in random hypergraph 2-coloring.
Proc.~23rd SODA (2012) 241--250.

\bibitem{DavisMcDonald}
B.~Davis, D.~McDonald:
An elementary proof of the local central limit theorem.
Journal of Theoretical Probability {\bf 8} (1995) 693--701.

\bibitem{DSS1}
J.~Ding, A.~Sly, N.~Sun:
Satisfiability threshold for random regular NAE-SAT.
Proc.\ 46th STOC (2014) 814--822.


\bibitem{DSS3}
J.~Ding, A.~Sly, N.~Sun:
Proof of the satisfiability conjecture for large $k$.
Proc.\ 47th STOC (2015) 59--68.

\bibitem{FriezeWormald}
A.~Frieze, N.{C.}~Wormald:
Random $k$-Sat: a tight threshold for moderately growing $k$.
Combinatorica {\bf 25} (2005) 297--305.

\bibitem{Janson}
S.~Janson: Random regular graphs: asymptotic distributions and contiguity.
Combinatorics, Probability and Computing {\bf 4} (1995) 369--405.


\bibitem{KKKS}
L.~Kirousis, E.~Kranakis, D.~Krizanc, Y.~Stamatiou:
Approximating the unsatisfiability threshold of random formulas.
Random Structures Algorithms {\bf 12} (1998) 253--269.

\bibitem{MPZ}
M.~M\'ezard, G.~Parisi, R.~Zecchina:
Analytic and algorithmic solution of random satisfiability problems.
Science {\bf 297} (2002) 812--815.

\bibitem{MRRW}
M. Molloy,  H. Robalewska, R.W. Robinson{, N.C. Wormald:}
1-factorisations of random regular graphs, {\it Random Structures and
Algorithms} {\bf
10} (1997), 305--321

\bibitem{R}
F. Rassmann: On the number of solutions in random hypergraph 2-colouring, arXiv:1603.07523 (2016).


\bibitem{Rathi}
V.~Rathi, E.~Aurell, L.~K.~Rasmussen, M.~Skoglund:
Bounds on threshold of regular random $k$-SAT.
Proc.\ 12th SAT (2010)  264--277.

\bibitem{RobinsonWormald}
R.~W.~Robinson{,} N.~C.~Wormald{:} \emph{Almost all cubic graphs are
hamiltonian}, Random Structures and Algorithms {\bf 3} (1992) 117--125.


\bibitem{SSZ}
A.\ Sly, N.\ Sun, Y.\ Zhang:
The number of solutions for random regular NAE-SAT.
arXiv:1604.08546 (2016).

\bibitem{Wth} N.C. Wormald{:} {\em Some Problems in the Enumeration of Labelled
Graphs}. Doctoral thesis, Newcastle University, 1978.
\end{thebibliography}
\end{document}